\documentclass[11pt]{article}

\usepackage{amsfonts, amsmath,amssymb,array,latexsym, hyperref, amscd}

\usepackage{fancyhdr,a4wide}
\usepackage{amsthm}
\usepackage{mathtools} 
\usepackage{subcaption}
\usepackage{xcolor}

\newtheorem{theorem}{Theorem}[section]
\newtheorem{lemma}[theorem]{Lemma}
\newtheorem{corollary}[theorem]{Corollary}
\newtheorem{proposition}[theorem]{Proposition}

\newtheorem{remark}[theorem]{Remark}

\numberwithin{equation}{section}

\newcommand{\eps}{\varepsilon}
\newcommand{\norm}[1]{\left|\left|#1\right|\right|}
\newcommand{\lr}[1]{\left(#1\right)}
\newcommand{\abs}[1]{\left|#1\right|}
\newcommand{\set}[1]{\left\{#1\right\}}
\newcommand{\E}[1]{\mathbb E\left[#1\right]}
\newcommand{\inprod}[2]{\left \langle #1,#2\right\rangle }
\def\Var{\mathrm{Var}}  
 
\newcommand{\x}{\times}
\newcommand{\R}{\mathbb R}
\newcommand{\C}{\mathbb C}
\newcommand{\gives}{\rightarrow}

\DeclareMathOperator{\TL}{TL}
\DeclareMathOperator{\vol}{vol}
\begin{document}

%\small

\title{Non-asymptotic Results for Singular Values of Gaussian Matrix Products}

\author{Boris Hanin\footnote{BH is supported by NSF grants DMS--1855684 and CCF--1934904},  Grigoris Paouris\footnote{GP is supported by NSF grants DMS--1812240 and CCF--1900929}\\
$\,\textsuperscript{*,\textdagger}$Dept. of Mathematics, Texas A\&M University, College Station, TX, 77843\\
$\,\textsuperscript{*}$Google, Mountain View, CA, 94043}

\maketitle

\begin{abstract}
This article provides a non-asymptotic analysis of the singular values (and Lyapunov exponents) of Gaussian matrix products in the regime where $N,$ the number of terms in the product, is large and $n,$ the size of the matrices, may be large or small and may depend on $N$. We obtain concentration estimates for sums of Lyapunov exponents, a quantitative rate for convergence of the empirical measure of the squared singular values to the uniform distribution on $[0,1]$, and results on the joint normality of Lyapunov exponents when $N$ is sufficiently large as a function of $n.$ Our technique consists of non-asymptotic versions of the ergodic theory approach at $N=\infty$ due originally to Furstenberg and Kesten \cite{furstenberg1960products} in the 1960's, which were then further developed by Newman \cite{newman1986distribution} and Isopi-Newman \cite{isopi1992triangle} as well as by a number of other authors in the 1980's. Our key technical idea is that small ball probabilities for volumes of random projections gives a way to quantify convergence in the multiplicative ergodic theorem for random matrices. 
\end{abstract}

\section{Introduction}
This article is about the spectral theory of random matrix products
\begin{equation}\label{E:X-def}
    X_{N,n} := A_N\cdots A_1,
\end{equation}
where $A_i$ are independent $n\x n$ matrices with independent real Gaussian entries $\lr{A_i}_{\alpha\beta}\sim \mathcal N(0,1/n)$ of mean zero and variance $1/n.$ We are primarily interested in the situation when $N$ is large and finite, while $n$ may depend on $N$ and may be either small or large. Our results concern the singular values of $X_{N,n}:$
\begin{equation}\label{E:sing-def}
s_1(X_{N,n})\geq \cdots \geq s_n(X_{N,n}),%\Sing_{N,n}:=\set{s_i\in [0,\infty)~|~ \det\lr{X_{N ,n}^*X_{N ,n}-s_i^2\Id}=0},~ 
\end{equation}
and can be summarized informally as follows:
\begin{enumerate}
\item We prove that as $N,n$ tend to infinity \textit{at any relative rate} the global distribution of the normalized squared singular values $\set{s_i(X_{N,n})^{2/N},\, i = 1,\ldots, n}$ converges to the uniform distribution on $[0,1]$ (see \S \ref{S:global-intro} and Theorem \ref{T:global}). Unlike previous results, we obtain quantitative concentration estimates valid for all $N,n$ larger than a fixed constant. See also \S \ref{S:why-triangle} for a heuristic explanation of why the uniform distribution appears in this context.
\item We prove that as long as $N$ is sufficiently large as a function of $n,$ the Lyapunov exponents
\begin{equation}\label{E:lyapunov-def}
    \lambda_i = \lambda_i(X_{N,n}) :=\frac{1}{N}\log s_i(X_{N,n})
\end{equation}
of $X_{N,n}$ are approximately independent and Gaussian (see Theorem \ref{T:normal} in \S \ref{S:normal}). Unlike previous results, our estimates simultaneously treat all the Lyapunov exponents and provide quantitative concentration estimates when $N$ is large but finite even when $n$ grows with $N.$
\item The statements listed above derived from our main technical result, Theorem \ref{T:sums}, which gives quantitative deviation estimates on sums of Lyapunov exponents of $X_{N,n}$:
\[\mathbb P\lr{\abs{\frac{1}{n}\sum_{i=m}^k(\lambda_i -\mu_{n,i})}\geq s}\leq c_1 e^{-c_2nNs\min\set{1, ng_{n,k}(s)}},\quad s\geq \frac{k}{nN}\log\lr{\frac{en}{k}},\]
where $\mu_{n,i}$ is defined in \eqref{E:mu-def} and $g_{n,k}(s)$ is a function defined in \eqref{E:g-def}. It is known that (e.g. equations (1) and (7) in \cite{newman1986distribution}) $\mu_{n,i}$ is the almost sure limit of $\lambda_i$ when $N\gives\infty$.
\end{enumerate}
In this article, we exclusively treat the case of $A_i$ having iid real Gaussian entries. This simplifies a number of arguments, but we conjecture that similar results hold if we assume only that the distribution of the entries of $A_i$ have finite fourth moments and bounded density. We leave this for future work. 

\subsection{Main Technical Result} 
Let us set some notation. Denote as in \eqref{E:lyapunov-def} by $\lambda_i=\lambda_i(X_{N,n})$ the Lyapunov exponents of $X_{N,n}$. Further, define for any $s>0$ \begin{equation}
\label{E:g-def}
g_{n,k}(s)  = \left\{ \begin{array}{cc} 
\min\left\{ 1, \frac{ n s}{ k}\right\} \  ,& \hspace{5mm} k\leq \frac{n}{2} \\
\min\left\{\delta_{n,k},  \frac{  s}{ \log{1/\delta_{n,k}}} \right\} \  ,& \hspace{5mm} \frac{n}{2}<k\leq n \\
\end{array} \right.,
\end{equation}
where for $ k \geq n/2$ we've set 
\[ \delta_{n,k}:= \frac{ n-k+1}{n}\in \left[ \frac{1}{n} , \frac{n-1}{n}\right].\]
Finally, write
\begin{equation}\label{E:mu-def}
    \mu_{n,k} := \E{\frac{1}{2}\log\lr{\frac{1}{n}\chi_{n-k+1}^2}}=\frac{1}{2}\lr{\log\lr{\frac{2}{n}}+\psi\lr{\frac{n-k+1}{2}}},
\end{equation}
where $\psi(z) = \frac{d}{dz}\log\Gamma(z)$ is the digamma function and $\chi_m^2$ is a chi-squared random variable with $m$ degrees of freedom. Our main technical result is the following

\begin{theorem}[Deviation Estimates for Sums of Lyapunov Exponents]\label{T:sums}
\label{T:sums}
\noindent There exist universal constants $c_1,c_2,c_3>0$ with the following property. Fix $ 1\leq m \leq k \leq n $ as well as $N\geq 1$. Then,
\begin{equation}
\label{E:sums}
\mathbb P\lr{\abs{\frac{1}{n}\sum_{i=m}^k (\lambda_i - \mu_{n,i})}\geq s} \leq c_2 \exp\lr{-c_3nNs\min\set{1,ng_{n,k}(s)}} {\color{purple} } ,%\exp{ \left\{ - c_3 n N s  g_{n,k}(s)  \right\}}.
\end{equation}
provided $s\geq c_1\frac{k}{nN}\log(en/k).$
 \end{theorem}
% \begin{remark}
%Some routine algebra reveals that if $N>Cn\log(n/k)$, then $nNs<n^2Nsg_{n,k}(s)$ and hence the estimate \eqref{E:sums} becomes:
% \begin{equation}
%\label{E:simple-sums}
%\mathbb P\lr{\abs{\frac{1}{n}\sum_{i=m}^k (\lambda_i - \mu_{n,i})}\geq s} \leq c_2 \exp{ \left\{ - c_3 n N s \right\}},\qquad  s\geq c_1 \frac{k}{nN}\log\lr{\frac{en}{k}}.
%\end{equation}
% \end{remark}

%{\color{red} 
%\begin{theorem}[Deviation Estimates for Sums of Lyapunov Exponents]\label{T:sums}
%\label{T:sums}
%\noindent There exist universal constants $c_1,c_2,c_3>0$ with the following property. Fix $ 1\leq m \leq k \leq n $ as well as $N\geq 1$. Then,
%\begin{equation}
%\label{E:sums}
%\mathbb P\lr{\abs{\frac{1}{n}\sum_{i=m}^k (\lambda_i - \mu_{n,i})}\geq s} \leq c_2 \exp{ \left\{ - c_3 n N s  g_{n,k}(s)  \right\}}.
%\end{equation}
%provided $s\geq c_1\frac{k}{nN}\log(en/k).$
% \end{theorem}
% \begin{remark}
% Note that if $s>Ck/(n\log(\delta_{n,k}))$, then $g_{n,k}(s)=1$, 
% then the estimate \eqref{E:sums} becomes:
% \begin{equation}
%\label{E:simple-sums}
%\mathbb P\lr{\abs{\frac{1}{n}\sum_{i=m}^k (\lambda_i - \mu_{n,i})}\geq s} \leq c_2 \exp{ \left\{ - c_3 n N s \right\}},\qquad  s\geq c_1 \frac{1}{n\log(\delta_{n,k})}.
%\end{equation}
% \end{remark} 
% }
Theorem \ref{T:sums} holds for every $n,N\geq 1$ and reveals a great deal about the singular values and Lyapunov exponents of $X_{N,n}$. For instance, in the bulk (i.e. when $k$ is comparable to $n$), the restriction on $s$ in \eqref{E:sums} reduces simply to $s> C/ N$, giving information about $X_{N,n}$ as soon as $N$ is large, regardless of $n$. This turns out to be  enough to prove Theorem \ref{T:global}, given in \S \ref{S:global-intro} below, which states that the squared singular values of $X_{N,n}$ approximate the uniform distribution on $[0,1]$ when $N,n$ tend to infinity at any relative rate. 

Theorem \ref{T:sums} also gives precise information about the top Lyapunov exponents of $X_{N,n}$. Indeed, taking $k$ to be fixed in \eqref{E:sums} gives non-trivial information on $\lambda_1,\ldots,\lambda_k$ as soon as $N\gg \log(n)$. Further, note that standard estimates for the digamma function $\psi$ yield
\begin{equation}\label{E:mu-est}
\mu_{n,k}= \log\lr{1-\frac{k-1}{n}} - \frac{1}{n-k+1} + O\lr{\frac{1}{(n-k+1)^2}}.
\end{equation}
This shows that the difference between the means $\mu_{n,1}$ and $\mu_{n,2}$ of $\lambda_1$ and $\lambda_2$ is on the order of $1/n$. As soon as $N\gg n\log(n)$, we may apply \eqref{E:sums} with $s\ll \lambda_1-\lambda_2$ to conclude that
\[ \frac{s_1(X_{N,n})}{s_2(X_{N,n})}=e^{N(\lambda_1-\lambda_2)}\geq e^{cN/n}\quad \text{with high probability}.\]
Hence, we find that $X_{N,n}$ begins to have a large spectral gap in the ``near ergodic'' regime $N\gg n\log(n)$. In fact, in Theorem \ref{T:normal}, we prove that in this regime $\lambda_1,\ldots,\lambda_k$ are also approximately independent Gaussians. We refer the reader to \S \ref{S:normal} for the details. 

A notable aspect of Theorem \ref{T:sums} is that it applies to any finite $n,N\geq 1$, allowing us to ``interpolate'' between the ergodic $N\gg n$ and free $n\gg N$ regimes. To explain this point, note that matrix products of the form \eqref{E:X-def} have been studied primarily in two setting. The first, which we refer to as the free probability regime occurs when $N$ is fixed and $n\gives \infty.$ This is a kind of maximum entropy regime in which the global distribution of singular values can be characterized in terms of maximizing the non-commutative entropy (cf eg \cite{arous1997large,banica2011free}). The second, which we call the ergodic regime, occurs when $n$ is fixed and $N\gives \infty$. This is a kind of minimal entropy regime in which the Lyapunov exponents (and singular values of $X_{N,n}$) tend to almost sure limits.

In both the ergodic and the free regimes, it is often difficult to obtain finite size corrections. Theorem \ref{T:sums} supplies such information. Moreover, since the ergodic and free regimes are usually treated by rather different means, it is unclear which techniques can give information that can interpolate between them. Our approach extends the ergodic techniques pioneered by Furstenburg-Kesten \cite{furstenberg1960products}, further developed in connection to random Schr\"odinger operators by Carmona \cite{carmona1982exponential} and Le Page \cite{le1982theoremes} (cf also \cite{bougerol1985concentration}), and applied in a very similar context as ours by Newman \cite{newman1986distribution} and Isopi-Newman \cite{isopi1992triangle}. It is therefore not surprising that in all of our results, we need $N$ to be in some sense large. 

Although we do not take this approach in the present article, it is also natural to study spectra of random matrix products by adapting techniques originally developed to treat the case when $N=1$. Indeed, in this setting, there has been considerable effort to obtain non-asymptotic analogs of classical random matrix theory results when $n=\infty$ \cite{rudelson2014lecture,vershynin12,rudelson2017delocalization}, culminating in the resolution of a number of long-standing open problems \cite{rudelson2008littlewood,rudelson2009smallest,adamczak2010quantitative,tikhomirov2020singularity}. More recently, several groups of authors \cite{henriksen2020concentration,huang2020matrix,kathuria2020concentration} have started to extend techniques for obtaining concentration for random matrices tailored (see \cite{tropp2015introduction}) to the small $N$ regime for understanding the kinds of matrix products considered in this article. From this point of view, our article takes a complementary approach, finding extensions of techniques originally coming from the ergodic theory used to analyze the case when  $N=\infty$.

\subsection{Convergence of Squared Singular Values to the Uniform Distribution}\label{S:global-intro}
Prior work \cite{isopi1992triangle,kargin2008lyapunov, tucci2010limits,gorin2018gaussian,liu2018lyapunov,ahn2019fluctuations} shows that in a variety of settings where $n,N\gives \infty$, the global distribution of singular values of $X_{N,n}^{1/N}$ converges to the so-called triangle law after proper normalization.  Informally, this means
\begin{equation}\label{E:triangle-informal}
    \lim_{N,n\gives \infty} \frac{1}{n}\#\set{j\leq n~|~ s_i^{1/N}(X_{N,n})\leq t} ~=~ \int_{-\infty}^t 2s{\bf 1}_{\set{s\in [0,1]}} ds~=:~TL(t).
\end{equation}
The graph of the density $2s{\bf 1}_{\set{s\in [0,1]}}$ of $\TL$ has the shape of a triangle, giving the distribution its name. With the exception of the articles \cite{liu2018lyapunov,gorin2018gaussian}, which obtain much more precise information for products of \textit{complex} Gaussian matrices and the article \cite{ahn2019fluctuations} concerning $\beta-$Jacobi products as well as the real Gaussian case, the majority of prior results about \eqref{E:triangle-informal} (e.g. \cite{isopi1992triangle,kargin2008lyapunov,tucci2010limits}) do not allow $n,N$ to tend to infinity simultaneously. Moreover, all prior results we are aware of do not give quantitative rates of convergence. Theorem \ref{T:global} provides both for the real Gaussian case we consider here. To state it, we note that if a random variable $T$ is distributed according to the triangle law, then $T^2$ is uniformly distributed on $[0,1]$, i.e. has the following cumulative distribution function:
\[\mathcal U(t):=\int_{-\infty}^t {\bf 1}_{[0,1]}(t)dt.\]

\begin{theorem}[Global Convergence to Triangle Law]\label{T:global}
There exist universal constants $c_1,c_2,c_3,c_4>0$ with the following property. For all $\eps\in (0,c_1)$, if $N>c_2/\eps^{2}$ and $n> c_3\log(1/\eps)/\eps$, then the probability 
\begin{align*}
&\mathbb P\lr{\sup_{t\in \R}\abs{{\frac{1}{n}\#\set{1\leq i\leq n~|~ s_i^{2/N}(X_{N,n})\leq t}} - \mathcal U(t)}\geq \eps }
\end{align*}
that the cumulative distribution for the squared singular values of $X_{N,n}$ deviates from the uniform distribution by more than $\eps$ is bounded above by 
\[ 4\exp\left[-c_4n N \varepsilon^{2}\min\set{1,ng_{n,k}(\varepsilon^{2})} \right].\]

\end{theorem}

%\begin{remark}\label{R:refinement}
%\noindent We can show that a weaker assumption on the dependence on $ \varepsilon , n$ is possible at the cost of a weaker probability. To be more precise, with $g_{n,k}(s)$ as in \eqref{E:g-def}, if  $N>c_3/\eps^{2}$ and $n>c_4 \frac{\sqrt{\log(1/\eps)}}{\eps}$,
%\[\mathbb P\lr{\sup_{t\in \R}\abs{{\frac{1}{n}\#\set{1\leq i\leq n~|~ s_i^{1/N}(X_{N,n})\leq t}} - TL(t)}\geq \eps}\leq 4\exp\left[-\frac{c_2n N \varepsilon^{2}}{ \delta_{n,k}} g_{n,k} ( \varepsilon^{2}/\delta_{n,k})\right].
%\]
%\end{remark}
In the next section we use the circular law \eqref{E:circular-law} for the (complex) \textit{eigenvalues} of $X_{N,n}^{1/N}$ to give an intuitive but heuristic explanation for why the uniform distribution (or equivalently the triangle law) should appear as the limiting distribution of singular values on $X_{N,n}$. Before doing so, we briefly discuss the dependence of Theorem \ref{T:global} on $N,n$, starting with the former. For fixed $N$, consider an iid random sequence $\set{X_{N,n}}_{n=1}^\infty$ with the product measure. Taking $\eps =2(c_2/N)^{1/2}=:CN^{-1/2}$, Theorem \ref{T:global} shows that
\[
\mathbb P \lr{\sup_{t\in \R}\abs{{\frac{1}{n}\#\set{1\leq i\leq n~|~ s_i^{2/N}(X_{N,n})\leq t}} - \mathcal U(t)}\geq \frac{C}{\sqrt{N}}} \leq 4e^{-cn},\quad c>0.
\]
Thus, by Borel-Cantelli, we find that
\begin{equation}\label{E:N-rate}
\sup_{t\in \R}\abs{\lim_{n\gives \infty}{\frac{1}{n}\#\set{1\leq i\leq n~|~ s_i^{2/N}(X_{N,n})\leq t}} - \mathcal U(t)}\leq \frac{C}{\sqrt{N}}\qquad \text{with probability }1,
\end{equation}
where by $\lim_{n\gives\infty} a_n$ we mean any limit point of the sequence $a_n$. This $1/\sqrt{N}$ can be seen as a Berry-Esseen-type estimate. To make this precise, consider
\[\rho_{N,\infty}:=\lim_{n\gives \infty}\frac{1}{n}\sum_{i=1}^n \delta_{s_{i}(X_{N,n})^{1/N}},\]
the large matrix limit for the empirical distribution of normalized singular values for $X_{N,n}.$ It is known \cite[Thm 6.1]{banica2011free} that 
\[\rho_{N,\infty} = \mathrm{qc}^{\boxtimes N},\qquad \mathrm{qc}(x):=\frac{1}{2\pi}\sqrt{x(2-x)}{\bf 1}_{\set{[0,2]}}(x),\]
where $\mathrm{qc}$ is the quarter circle law and $\boxtimes$ is the multiplicative free convolution. Kargin \cite{kargin2008lyapunov} and Tucci \cite{tucci2010limits} show that, consistent with Theorem \ref{T:global}, 
\[\lim_{N\gives \infty} \rho_{N,\infty} = \mathrm{TL}.\]
As far as we know, the optimal rate of convergence for such repeated multiplicative free convolution is unknown. However, from this point of view, \eqref{E:N-rate} shows that the rate of convergence is at least as fast as in the usual central limit theorem. 

To understand the dependence of Theorem \ref{T:global} on $n,$ we send $N$ to infinity in Theorem \ref{T:global} to obtain as before that there is $C>0$ so that
\begin{equation}\label{E:n-rate}
\sup_{t\in \R}\abs{\lim_{N\gives \infty}{\frac{1}{n}\#\set{1\leq i\leq n~|~ s_i^{2/N}(X_{N,n})\leq t}} - \mathcal U(t)}\leq \frac{C\log(n)}{n}\qquad \text{with probability }1,
\end{equation}
Apart from the $\log(n)$, this estimate is sharp. Indeed, the empirical distribution
\[\rho_{\infty,n}:=\lim_{N\gives \infty}\frac{1}{n}\sum_{i=1}^n \delta_{s_{i}(X_{N,n})^{1/N}}\]
of singular values in the large number of matrices limit exists almost surely and is deterministic by the Multiplicative Ergodic Theorem. Among other things, Theorem \ref{T:normal} below computes, in agreement with the early work of Newman \cite{newman1986distribution}, this limit in our Gaussian case. The subsequent work of Isopi-Newman \cite{isopi1992triangle} showed that, under minimal assumptions,
\[\lim_{n\gives \infty}\rho_{\infty,n}=\mathrm{TL}.\]
This of course agrees with Theorem \ref{T:global}, which via \eqref{E:n-rate} provides a natural rate of convergence. This rate is optimal, perhaps up to the $\log(n),$ because the spacing of the atoms in $\rho_{\infty,n}$ is approximately $1/n$. Hence, the distance between $\rho_{\infty,n}$ and triangle law $\mathrm{TL}$, which is a continuous distribution, is bounded below by a constant times $1/n$. 

\subsection{Why the Uniform Distribution in Theorem \ref{T:global}?}\label{S:why-triangle}
A number of articles \cite{newman1986distribution,isopi1992triangle,kargin2008lyapunov,tucci2010limits,liu2018lyapunov, gorin2018gaussian} show as in \eqref{E:triangle-informal} that in the limit where $n,N$ tend to infinity, the singular values  $s_i(X_{N,n})^{1/N}$ (or for similar matrix products) converge to the triangle law (and hence their squares converge to the uniform distribution on $[0,1]$). These articles use a variety of techniques ranging from free probability to ergodic theory and special functions. Why does the uniform distribution appear? The purpose of this section to give an intuitive explanation for this phenomenon. After writing an initial draft of this article,  we learned from G. Akemann that an explanation similar to the one below can be found on pages 3,4 in \cite{akemann2014universalb}. We also refer the reader to the work of Kieberg-K\"osters \cite{kieburg2016exact} about an exact relation between eigenvalues and singular values for products of complex Ginibre matrices.

Since $X_{N,n}$ is not normal with probability $1$, its spectral properties are captured not only its singular values but also by its eigenvalues
\begin{equation}\label{E:spec-def}
\abs{\zeta_1(X_{N,n})}\geq \cdots \geq \abs{\zeta_n(X_{N,n})},\qquad \zeta_i(X_{N,n})\in \C.
\end{equation}
Our argument for why the triangle law appears in Theorem \ref{T:global} relates the singular values and eigenvalues of $X_{N,n}$ and consists of two observations. First, consider the (complex) eigenvalues of $X_{N,n}^{1/N}$ as defined in \eqref{E:spec-def}. It is shown in \cite{gotze2010asymptotic, o2011products} that for each fixed $N$ the empirical distribution of the eigenvalues of $X_{N,n}^{1/N}$ converges weakly almost surely to the uniform measure on the unit disk in $\C.$ This result is often called the circular law. Informally, it reads
\begin{equation}\label{E:circular-law}
\lim_{n\gives \infty}\frac{1}{n} \sum_{i=1}^{n} \delta_{\zeta_i^{1/N}}(z)~~=~~ \frac{1}{\pi}{\bf 1}_{\set{\abs{z}\leq 1}},\qquad z\in \C
\end{equation}
Precise results on the rate of convergence can be found in \cite{gotze2018rate,jalowy2019rate} and local limit theorems are obtained in \cite{nemish2017local}. Since in polar coordinates $(r,\theta)$ the radial part of the uniform measure on the unit disk is $2rdr$, a corollary of the circular law is that
\begin{equation}\label{E:evals-TL}
    \text{For }N\text{ fixed, as }n\gives \infty,\text{ squared eigenvalue moduli }\abs{\zeta_i}^{2/N}\text{ of }X_{N,n}^{1/N}\text{ converge to }\mathcal U.
\end{equation}
Thus, the uniform distribution $\mathcal U$ appears naturally as the distribution of the squared moduli of eigenvalues of $X_{N,n}^{1/N}$ for every $N!$ On the other hand, it has been proved that for any fixed finite $n$ \cite{reddy2016lyapunov,reddy2019equality} that when $N$ is large
\[\forall i=1,\ldots, n \qquad \abs{\zeta_i}^{1/N}\approx s_i^{1/N}.\]
Thus, we extract another piece of intuition:
\begin{equation}\label{E:singval-eval}
    \text{For }n\text{ fixed, as }N\gives \infty,\text{ eigenvalue moduli and singular values of }X_{N,n}^{1/N}\text{ coincide}.
\end{equation}
Putting together \eqref{E:evals-TL} and \eqref{E:singval-eval}, we conclude heuristically that if both $n,N$ tend to infinity then the distribution of the singular values $s_{i}^{1/N}$ should converge to the triangle law. This is precisely the content of Theorem \ref{T:global}. While the heuristic for \eqref{E:singval-eval} was previously established only when $n$ is fixed, we believe it can also be proved in the regime where $n$ is allowed to grow with $N$ but leave this for future work.

\subsection{Distribution of Lyapunov Exponents in the Near Ergodic Regime}\label{S:normal}
In addition to studying the global distribution of singular values of $X_{N,n}$, we also obtain in Theorem \ref{T:normal} precise estimates for the joint distribution of the Lyapunov exponents 
\begin{equation}\label{E:lambda-def}
\lambda_i=\lambda_i(X_{N,n})=\frac{1}{N}\log s_i(X_{N,n})
\end{equation}
of $X_{N,n}$ in the regime when $N\gg n\log^2(n).$ To state it, we need some notation. Recall first that for each $1\leq  k \leq n$ we had set
\begin{equation}\label{E:mu-def-second}
    \mu_{n,k} = \E{\frac{1}{2}\log\lr{\frac{1}{n}\chi_{n-k+1}^2}}=\frac{1}{2}\lr{\log\lr{\frac{2}{n}}+\psi\lr{\frac{n-k+1}{2}}},
\end{equation}
where $\psi(z) = \frac{d}{dz}\log\Gamma(z)$ is the digamma function and $\chi_m^2$ is a chi-squared random variable with $m$ degrees of freedom. We also recall the estimate \eqref{E:mu-est}:
\begin{equation*}
\mu_{n,k}= \log\lr{1-\frac{k-1}{n}} - \frac{1}{n-k+1} + O\lr{\frac{1}{(n-k+1)^2}}.
\end{equation*}
The quantity $\mu_{n,k}$ already appears in \cite{newman1986distribution,isopi1992triangle} as the mean of $\lambda_k$ when $N\gives \infty.$ We futher define
\begin{equation}\label{E:sigma-est}
    \sigma_{n,k}^2:=\Var\left[\frac{1}{2}\log\lr{\frac{1}{n}\chi_{n-k+1}^2}\right]= \psi'\lr{\frac{n-k+1}{2}} = \frac{1}{2(n-k+1)}+O\lr{\frac{1}{(n-k+1)^2}},
\end{equation}
and set
\begin{equation}\label{E:mu-sigma-def}
    \mu_{n,\leq k}:= \lr{\mu_{n,1},\ldots, \mu_{n,k}},\qquad \sigma_{n,\leq k}^2:=\lr{\sigma_{n,1}^2,\ldots, \sigma_{n,k}^2}.
\end{equation}
Finally, we will consider for two $\R^k$-valued random variables $X,Y$ the following high-dimensional generalization of the usual Kolmogorov-Smirnov distance:
\begin{equation}\label{E:dist-def}
d( X, Y ) := \sup_{ C \in {\cal{C}}_{k} } \left| \mathbb P ( X\in C ) -\mathbb P ( Y\in C) \right|,
\end{equation}
where $ {\cal{C}}_{k}$ is the collection of all convex subsets of $\mathbb R^{k}$. 
\begin{theorem}[Asymptotic Normality of Lyapunov Exponents]\label{T:normal}
There exist constants $C_1,C_2>0$ with the following property. Suppose $X_{N,n}$ is as in \eqref{E:X-def}, fix $1\leq k \leq n$, and write
\[\Lambda_k = \lr{\lambda_1,\ldots, \lambda_k}\]
for the vector of the top $k$ Lyapunov exponents of $X_{N,n}.$ Then, $\lambda_1,\ldots, \lambda_k$ are approximately independent and Gaussian when $N$ is sufficiently large as a function of $k,n$:
\begin{equation}\label{E:dist-est}d\lr{\Lambda_k, \,\mathcal N\lr{\mu_{n,\leq k},\, \frac{1}{N}\mathrm{Diag}\lr{\sigma_{n,\leq k}^2}}} \leq C_2\lr{\frac{k^{7/2}n\log^2(n)\log^2(N/n)}{N}}^{1/2}.
\end{equation}
Here $\mathcal N(\mu, \Sigma)$ denotes a Gaussian with mean $\mu$ and co-variance $\Sigma$ and for any $v=\lr{v_1,\ldots,v_k}\in \R^k$ we have written $\mathrm{Diag}(v)$ for the diagonal matrix with $\mathrm{Diag}(v)_{ii}=v_i.$
\end{theorem}
\begin{remark}
The arguments in \cite{akemann2012universal,akemann2012universal,akemann2014universalb, akemann2019integrable} strongly suggest (see \S \ref{S:intution-IPS}) that for $k$ fixed and independent of $n$, a necessarily and sufficient condition for $\lambda_1,\ldots, \lambda_k$ to be close to independent and Gaussian is $N\gg n.$ Thus, the $\log^2(n)\log^2(N/n)$ in \eqref{E:dist-est} is likely sub-optimal. It is not clear whether the power $k^{7/2}$ can be improved. 
\end{remark}

For $k\geq 1$ fixed independent of $n,N$, Theorem \ref{T:normal} shows that the top $k$ Lyapunov exponents of $X_{N,n}$ are close to independent Gaussian as soon as $N\gg n\log^2(n)\log^2(N/n).$ This is a significant refinement of the result in \cite{carmona1982exponential} (see also Theorem 5.4 in \cite{bougerol1985concentration}), which states that when $n$ is fixed $\lambda_1$ is asymptotically normal. It also refines the recent result of Reddy \cite[Theorem 11]{reddy2019equality}, which holds only for fixed finite $n$ and does not give estimates at finite $N$. The advantage of Theorem \ref{T:normal} is that it treats simultaneously any number of Lyapunov exponents and gives a rate of convergence. For example, taking $k=n,$ we find that if $N\gg n^{9/2}\log^2(n)\log^2(N/n)$, then \textit{all} Lyapunov exponents of $X_{N,n}$ are approximately independent Gaussians. However, results in articles such as \cite{carmona1982exponential} are for matrix products $A_N\cdots A_1$ in which the entries of $A_i$ have mean zero, variance $1/n$ and satisfy some mild regularity assumptions, whereas our results hold only for the Gaussian case. We conjecture that Theorem \ref{T:normal} holds in this more general setting as well but leave this to future work. 

\section{Prior Work and Intuitions}
The purpose of this section is to give an exposition of prior work and provide  several intuitions for thinking about the matrix products $X_{N,n}$, especially about the differences between the near-ergodic $N\gg n$ and the near-free $n\gg N$ regimes. We do this by first giving in \S \ref{S:intuition-dynamics} a basic intuition from dynamical systems, which suggests that one can think of $N$ as a time variable and $n$ as a system size. This intuition dovetails with the multiplicative ergodic theorem. We proceed in \S \ref{S:intution-IPS} to explain an exact correspondence derived in \cite{akemann2012universal,akemann2014universal,akemann2014universalb,akemann2019integrable} at a physical level of rigor in which $n/N$ plays the role of a time parameter for the evolution of the $n$ singular values of $X_{N,n}$. This helps to explain why even simple linear statistics behave differently depending on the relative size of $n,N$.

\subsection{$X_{N,n}$ at Fixed $n$ as a Dynamical System}\label{S:intuition-dynamics}
One way to intuitively think of $X_{N,n}=A_N\cdots A_1$ is as defining the time $0$ to time $N$ map for a dynamical system in which the time one dynamics are very chaotic and are modelled as multiplication by an iid random matrix. In this analogy, $N$ takes on the role of a time parameter, whereas $n$ denotes the system size. Since large systems take longer to come to equilibrium, we should expect that $N$ and $n$ are ``in tension.'' If we fix $n$ and let $N$ tend to infinity, then the size of the long time image $\norm{X_{N,n}u}$ of an unit length input $u\in \R^n$ satisfies a pointwise ergodic theorem: \begin{equation}\label{E:LLN}
\lim_{N\gives \infty}\frac{1}{N}\log \norm{X_{N,n}u} = E,
\end{equation}
where $E$ is a constant (independent of $u$) depending on the measure $\mu$ according to which the entries of the matrices $A_i$ making up the matrix product $X_{N,n}$ are distributed. This can be proved in a variety of ways (e.g. Corollary 3.2 in \cite{cohen1984stability}). In fact, much more is true. It was shown by Kesten-Furstenberg in \cite{furstenberg1960products}, that this statement tolerates taking a supremum over $u$:
\[\lim_{N\gives \infty} \lambda_1(X_{N,n}) = E\qquad \text{almost surely}.\]
Later, in his seminal work \cite{oseledets1968multiplicative} Oseledets proved the multiplicative ergodic theorem. In the context of iid products of $N$ matrices of size $n\x n$, it says that under some mild conditions on $\mu$ if $n$ is fixed, then the full list of Lyapunov exponents $\lambda_1(X_{N,n}),\ldots, \lambda_n(X_{N,n})$ converges almost surely to a deterministic limit. We refer the reader to \cite{filip2019notes} for a review of the vast literature on this subject and to \cite{bougerol1985concentration} for an exposition specifically about matrix products.

Determining the values of the limiting Lyapunov exponents in the multiplicative ergodic theorem is in general quite difficult and has applications to Anderson localization for random Schr\"odinger operators \cite{bougerol1985concentration,damanik2011short}. 

Moreover, the work of LePage \cite{le1982theoremes} as well as subsequent analysis \cite{carmona1982exponential,bougerol1985concentration} showed that the top Lyapunov exponent of matrix products such as $X_{N,n}$ (not necessarily Gaussian) is asymptotically normal in the sense that there exist $a_{n},b_{N,n}\in \R$ so that
\[b_{N,n}\lr{\lambda_1(X_{N,n})- a_{n}}\quad \stackrel{d}{\longrightarrow}\quad  \mathcal N(0,1),\]
where the $d$ indicates that the convergence is in the sense of distribution. As far as we are aware, all known mathematical proofs of asymptotic normality results hold only for finite fixed $n$, for the top Lyapunov exponent $\lambda_1$ and do not include quantitative rates of convergence. For the real Gaussian case we study, our Theorem \ref{T:normal} overcomes these deficiencies. However, at the physical level of rigor, we refer the reader to the excellent articles \cite{akemann2012universal,akemann2014universal,akemann2014universalb,akemann2015recent, akemann2019integrable} that derive in the case of complex Gaussian matrix products asymptotic normality and much more for the top Lyapunov exponents (cf \S \ref{S:intution-IPS}). 

While the preceding discussion concerned matrix products with any entry distribution $\mu$ with mean $0$ and variance $1/n$, the Gaussian $\mu= \mathcal N(0,1/n)$ considered in this article leads to some significant simplifications. For instance, Newman \cite{newman1986distribution} computed the exact expression, which can be written in terms of the digamma function, for the limiting Lyapunov exponents. Similarly, \eqref{E:LLN} is a simple fact in this case since $\frac{1}{N}\log \norm{X_{N,n}u}$ turns out to be sum of iid random variables (see Lemma \ref{L:pointwise-iid}). These simplifications stem from the fact that the distribution of each matrix $A_i$ is left and right-invariant under multiplication by an orthogonal matrix.

\subsection{$n/N$ as a Time Parameter in an Interacting Particle System}\label{S:intution-IPS} 
In the regime where $n/N$ is bounded away from $0$ and $\infty$ as $n,N\gives \infty$, even the behavior of an innocuous seeming log-linear statistic depends very much on the ratio of $n$ and $N$. Informally, 
\begin{equation}\label{E:linear-stats-intro}
    \log \norm{X_{N,n}u}~\approx~ \mathcal N\lr{-\frac{N}{4n},\frac{N}{4n}} + O\lr{\frac{N}{n^2}}, \qquad u\in \R^n,\,\, \norm{u}=1,
\end{equation}
where $\mathcal N(\mu, \Sigma)$ denotes a Gaussian with mean $\mu$ and covariance $\Sigma.$ As mentioned above, in the Gaussian case we consider in this article, this approximate normality is easy to see since $\log \norm{X_{N,n}u}$ is a sum of iid variables (see Lemma \ref{L:pointwise-iid}). Precise versions of \eqref{E:linear-stats-intro} also hold true when the matrices $A_i$ in the definition \eqref{E:X-def} of $X_{N,n}$ have symmetric but non-Gaussian entries (see Theorem 1 in \cite{hanin2019products}). It is interesting to compare the almost sure convergence to a constant in \eqref{E:LLN} (note that $1/N$ normalization) with the asymptotic normality in \eqref{E:linear-stats-intro}.

The relation \eqref{E:linear-stats-intro} already suggests that $t=n/N$ is an important parameter for interpolating between the ergodic regime, defined by $t=0$ and the asymptotically free regime, in which $t=\infty$. A number of remarkable articles \cite{akemann2012universal,akemann2014universal,akemann2014universalb} and especially \cite{akemann2019integrable} establish a correspondence between $t$ and the time parameter in the stochastic evolution of an interacting particle system. {This correspondence between singular values for products of complex Ginibre matrices and DBM appears to be initially due to Maurice Duits.} 

The particles in question are the limiting Lyapunov exponents $\lambda_i$ of $X_{N,n}$. When $t=0$, they are approximately uniformly spaced (see Theorem \ref{T:normal}) and are interpreted as an initial condition for Dyson Brownian motion (DBM)
\begin{equation}\label{E:DBM}
d\lambda_i = dB_i + \sum_{j\neq i} \frac{dt}{\lambda_i-\lambda_j},\qquad i=1,\ldots,n
\end{equation} 
the dynamics induced on the spectrum of a matrix by allowing each entry to evolve for time $t$ under and independent Brownian motion \cite{dyson1962statistical}. The surprising observation is that, at least in the bulk of the spectrum (i.e. $\lambda_k$ with $k$ proportional to $n$) the joint distribution of the Lyapunov exponents of $X_{N,n}$ satisfies \eqref{E:DBM} at time $t$ with an equally spaced initial condition in the limit when $n/N=t$ and $n,N\gives \infty.$ 

The idea of the derivations in  \cite{akemann2012universal,akemann2014universal,akemann2014universalb, akemann2019integrable} is to use that when $X_{N,n}$ is a product of complex Ginibre matrices, the joint distribution of all of its singular values, at any finite $n,N$, is given by a determinental point process. One may then study the scaling limit of the corresponding determinental kernel at any fixed $t=n/N$. This kernel coincides with the solution to DBM from equally spaced initial conditions, which is also determinental  \cite{johansson2004determinantal}.

A rigorous analysis of the determinental kernel for the joint distribution of singular values for products of complex Gaussian matrices was undertaken in a variety of articles \cite{forrester2013lyapunov, forrester2014eigenvalue, forrester2016singular, liu2016bulk,liu2018lyapunov}. In particular, \cite{liu2016bulk} shows that when $N$ is arbitrary but fixed and $n\gives \infty$, the determinental kernel for singular values in products of $N$ iid complex Gaussian matrices of size $n\x n$ converges to the familiar sine and Airy kernels that arise in the local spectral statistics of large GUE matrices in the bulk and edge, respectively. This agrees with the prediction from \cite{akemann2019integrable}. Indeed, in this regime, the time parameter $t=n/N$ is infinite and the limiting distribution of DBM is that of the  eigenvalues for a large GUE matrix. Moreover, \cite{liu2018lyapunov} rigorously obtained an expression for the limiting determinental kernel when $t=n/N$ is arbitrary in the context of products of complex Ginibre matrices.{We refer the reader also to the subsequent article of Liu-Wang \cite{liu2019phase} that performs a similar analysis for the \textit{eigenvalues} in the same setting. }

Also in the regime where $n/N$ is fixed while $n,N\gives \infty$, we refer the reader to Gorin-Sun \cite{gorin2018gaussian}. This article shows that the fluctuations of the singular values of $X_{N,n}$ around the triangle law always converge to a Gaussian field. We also refer the reader to \cite{ahn2019fluctuations}, which obtains a CLT for linear statistics of top singular values when $n/N$ is fixed and finite.

\section{Idea of Proof: Reduction to Small Ball Estimates}
Before turning to the formal proofs of Theorems \ref{T:global} and \ref{T:normal}, we give a brief overview of our approach, which begins with the following representation (cf e.g. \cite{newman1986distribution,isopi1992triangle}) for sums of Lyapunov exponents from Lemma \ref{lem-Grass-3} (see \S \ref{S:wedge-background}):
\begin{equation}\label{E:wedge-rep-informal}
\lambda_1+\cdots + \lambda_k = \sup_{\substack{\Theta \in \mathrm{Fr}_{n,k}}} \frac{1}{N}\log \norm{X_{N,n}(\Theta)}.
\end{equation}
In the previous line, we've denoted by $\mathrm{Fr}_{n,k}$ the collection of all orthonormal $k-$frames in $\R^n$ (i.e. collections of $k$ orthonormal vectors $v_1,\ldots, v_k$). We have also set
\[X_{N,n}(\Theta)=X_{N,n}\theta_1\wedge\cdots\wedge X_{N,n}\theta_k,\qquad \Theta = (\theta_1,\ldots,\theta_k)\in \mathrm{Fr}_{n,k},\]
where we recall that for $a,b\in \R^n$ $a_1\wedge \cdots \wedge a_k$ is the anti-symmetrization of $a_1\otimes \cdots \otimes a_k$. We refer the reader to \S \ref{S:wedge-background} for more background and to the start of \S 7 in \cite{spivak1970comprehensive} and to \S 2.6.1 in \cite{tao2010epsilon} for more on wedge products. As pointed out in \cite{isopi1992triangle}, information about the sums $\lambda_1+\cdots+\lambda_k$ can easily be translated into the information about their cumulative distribution function, ultimately resulting in Theorem \ref{T:global}. Similarly, the vector of the top $k$ Lyapunov exponents considered in Theorem \ref{T:normal} can be obtained by an affine transformation of the vector of partial sums $\lambda_1,\lambda_1+\lambda_2,\ldots,\lambda_1+\cdots + \lambda_k$. Thus, the focus of our proofs is to obtain precise concentration estimates for the expression on the right hand side of \eqref{E:wedge-rep-informal}. An important idea for analyzing \eqref{E:wedge-rep-informal}, which goes back to the work of Furstenburg-Kesten \cite{furstenberg1960products} is that when $N$ is large, one can almost drop the supremum:
\begin{equation}\label{E:pointwise-informal}
\lim_{N\gives \infty}\abs{\frac{1}{N}\log\norm{X_{N,n}(\Theta)}-\sup_{\substack{\Theta' \in \mathrm{Fr}_{n,k}}} \frac{1}{N}\log\norm{X_{N,n}(\Theta')} } = 0
\end{equation}
for any fixed $\Theta\in \mathrm{Fr}_{n,k}$. As explained below this is plausible since the ratio $s_k(X_{N,n})/s_{k+1}(X_{N,n})$ of the $k^{th}$ and $(k+1)^{st}$ singular values grows exponentially with $N$, causing the wedge product $X_{N,n}(\Theta)$ to align almost entirely with the wedge product of the top $k$ singular vectors of $X_{N,n}$ for almost every $\Theta$. 

The ``pointwise'' quantity $\frac{1}{N}\log \norm{X_{N,n}(\Theta)}$ is a sum of iid random variables (Lemma \ref{L:pointwise-iid}) and can be analyzed using a result of \L ata\l a \cite{latala1997estimation} (see Theorem \ref{Latala}). It then remains to obtain quantitative versions of \eqref{E:wedge-rep-informal} valid for large but finite $N,n$. One possible approach is via energy-entropy estimates using $\eps$-nets on $\mathrm{Fr}_{n,k}$. However, while this gives some results, this approach is suboptimal for large $N$. The reason that $\eps$-nets fail is that, due to the $N^{-1}$ normalization,
\[N~~ \text{large} \quad \Rightarrow\quad \Var_\Theta\left[\frac{1}{N}\log \norm{X_{N,n}(\Theta)}\right]\ll \Var_{X_{N,n}}\left[\frac{1}{N}\log \norm{X_{N,n}(\Theta)}\right]\]
by which we mean that the variance of $\frac{1}{N}\log \norm{X_{N,n}(\Theta)}$ over $\Theta\in \mathrm{Fr}_{n,k}$ for a typical realization of $X_{N,n}$ is much smaller than its variance over the randomness in $X_{N,n}$ for any fixed $\Theta$, causing the optimal net to have constant cardinality.  %Put another way, the space $\mathrm{Fr}_{n,k}$ is independent of $N$. Hence, due to ergodicity \eqref{E:pointwise-informal}, its $\epsilon$-entropy will always be larger than the energy for the pointwise concentration of $N^{-1}\log\norm{X_{N,n}(\Theta)},$ 

The main technical novelty of our proofs is that we  quantify \eqref{E:pointwise-informal} not through net arguments but rather via small ball probabilities for volumes of random projections, which are already known (cf Proposition \ref{Gaussian-matrix-det-est}). The key result is the following: 
\begin{proposition}\label{P:small-ball-bounds}
For any $\eps\in (0,1)$ and any $\Theta \in \mathrm{Fr}_{n,k}$ we have
 \[\mathbb P\lr{\abs{\frac{1}{N}\log \norm{X_{N,n}\Theta} -\sup_{\Theta' \in \mathrm{Fr}_{n,k}}\frac{1}{N}\log \norm{X_{N,n}\Theta'}}\geq \frac{1}{N}\log\lr{\frac{1}{\eps}}}\leq \mathbb P\lr{\norm{P_F(\Theta)}\leq \eps},\]
 where $F$ is a Haar distributed $k-$dimensional subspace of $\R^n$ and 
 \[P_F(\Theta)=P_F\theta_1\wedge \cdots\wedge P_F\theta_k\] 
 with $P_F$ denoting the orthogonal onto $F.$
\end{proposition}
For the proof of Proposition \ref{P:small-ball-bounds} see Lemma \ref{L:small-ball-bounds}. The appearance of small ball probabilities is natural, although perhaps somewhat unexpected. Let us briefly describe why in the simplest case of $k=1$. Denote by $v^{(i)}$ the right eigenvector of $X_{N,n}$ corresponding to the singular value $s_{i}(X_{N,n})$. For any $\theta \in S^{n-1}$, we may write
\[\norm{X_{N,n}\theta}^2 = \sum_{i=1}^n \abs{\inprod{X_{N,n}\theta}{v^{(i)}}}^2.\]
When $N\gg n$, the matrix $X_{N,n}$ is highly degenerate in the sense that there exists a universal constant $C>0$ so that
\[\frac{s_{1}(X_{N,n})}{s_{2}(X_{N,n})}\geq e^{CN/n}.\]
This is easy to see intuitively since in this regime $\lambda_1-\lambda_2 \approx \frac{1}{n}$ (cf Theorem \ref{T:normal}). Hence, 
\begin{equation}\label{E:degen}
    \norm{X_{N,n}\theta}^2 \approx  \abs{\inprod{X_{N,n}\theta}{v^{(1)}}}^2
\end{equation}
unless $\abs{\inprod{\theta}{v^{(1)}}}$ is unusually small. In fact, for all $\theta$ 
\[0\geq \frac{1}{N}\log\norm{X_{N,n}\theta}-\lambda_1 = \frac{1}{2N}\log \lr{\frac{\norm{X_{N,n}\theta}^2}{s_1^2(X_{N,n})}}\geq \frac{1}{N}\log\abs{\inprod{\theta}{v^{(1)}}}.\]
This lower bound is essentially sharp by \eqref{E:degen} unless $\theta$ has small overlap with $v^{(1)}$, an event whose probability is controlled precisely by a small ball estimate. 

%In particular, note that $v^{(1)}$ is uniform on $S^{n-1}$ by the rotational invariance of $X_{N,n}$. Thus, using standard small ball estimates for the first coordinate of a Haar distributed vector on $S^{n-1}$, we find for some $C>0$ that for any $\eps\in (0,1)$
%\[\mathbb P\lr{\abs{\frac{1}{N}\log\norm{X_{N,n}\theta}-\lambda_1}\leq \frac{C}{N}\log\lr{\frac{n}{\eps^2}}}\geq 1- \eps,\]
%giving information about $\lambda_1$ already when $N\gg\log(n).$ Generalizations of this argument (see Corollary \ref{C:small-ball} in \S \ref{S:small-ball-use}) make precise \eqref{E:pointwise-informal} and reduce Theorems \ref{T:global} and \ref{T:normal} to obtaining sharp concentration of $\log\norm{ X_{N,n}(\Theta)}$ for any fixed $\Theta$. In Gaussian case we consider here, $\log\norm{ X_{N,n}(\Theta)}$ is a sum of iid random variables (see Lemma \ref{L:pointwise-iid}), a fact already noted by Newman in \cite{newman1986distribution}. The exact behavior of its moments (and hence tail behavior) are given by a Theorem of \L ata\l a (see Theorem \ref{Latala} below) in combination with a somewhat technical computation in \S \ref{S:pointwise-pf}. 

\section{Acknowledgements} We are grateful to Vadim Gorin, Maurice Duits, and Gernot Akemann for pointing us to a number of interesting references. We would also like to thank a referee for a very careful and helpful reading of an earlier draft that pointed out a number of inaccuracies and ultimately lead to a substantially improved exposition.

\section{Organization of the Rest of the Article} The rest of this article is structured as follows. First, in \S \ref{S:wedge-background} we collect some well-known results on the relation between the exterior algebra of $\R^n$ and the singular values of any linear map $A:\R^n\gives \R^n.$ We also record several elementary observations (Lemmas \ref{L:polar} and \ref{L:haar-flags}) about polar decompositions and Haar measures on orthonormal frames and their flags. We will use this formalism throughout our proofs. 

Next, in \S \ref{S:iid-background} recalls two kinds of results. The first, Theorem \ref{Latala}, is a result of \L ata\l a \cite{latala1997estimation} that gives  precise information on moments (and hence tail behavior) for sums of independent centered random variables. The second is a set of results related to the multivariate central limit theorem (Theorem \ref{MultiCLT-theo}) and the Gaussian content of boundaries of convex sets (Theorem \ref{T:Nazarov-Ball}). The latter allows us to prove Proposition \ref{Dist-lem}, a stability result for the Kolmogorov-Smirnov-type distance function $d$ used in the statement of Theorem \ref{T:normal}. Section \ref{S:roadmap} follows, containing a brief road map to the proofs of Theorems \ref{T:sums}, \ref{T:global} and \ref{T:normal}. Then, \S \ref{S:small-ball-use} is devoted to explaining how to use small ball estimates on volumes of random projections to formalize the ergodicity \eqref{E:pointwise-informal}.

Further, the results in \S \ref{S:concentration} are used in all our proofs. The main result there is Proposition \ref{P:pointwise}, which together with Proposition \ref{P:pointwise-small-ball} and Lemma \ref{L:pointwise-iid} explains the appearance of chi-squared random variables in the statement of Theorem \ref{T:normal}. The proof of Proposition \ref{P:moments} is the most technical part of our arguments. Next, we complete the proof of Theorem \ref{T:sums} in \S \ref{S:sums-pf}. We then use Theorem \ref{T:sums} to complete in \S \ref{S:global-pf} and \S \ref{S:normal-outline} the proofs of Theorems \ref{T:global} and \ref{T:normal}, respectively.

\section{Singular Values via Wedge Products}\label{S:wedge-background} 
In this section, we recall some background on wedge products and refer the reader to the start of \S 7 in \cite{spivak1970comprehensive} and to \S 2.6.1 in \cite{tao2010epsilon} for more details. The usual $\ell_2$-structure on $\R^n$ gives rise in a functorial way to an $\ell_2$ structure on the exterior powers $\Lambda^k \R^n$. If $ x_{1}, \cdots , x_{k} $ are in $ \mathbb R^{n}$ (e.g. are a frame for an element of $G_{n,k}$) we denote the resulting norm by 
\[\norm{x_1\wedge\cdots\wedge x_k}.\]
If we denote by $ X^{\ast}$ the $ n\times k$ matrix $( x_{1}, \cdots , x_{k})$, the Gram identity reads
\begin{equation}
\label{wedge-def}
 \| x_{1} \wedge \cdots \wedge x_{k} \|  = \sqrt{{\rm det} ( X X^{\ast})} = \vol_k\lr{\mathcal P \lr{x_1,\ldots, x_k}},
\end{equation}
where $\mathcal P \lr{x_1,\ldots, x_k}$ is the parallelopiped spanned by $x_1,\ldots, x_k$. The following Lemma gives a well-known  characterization of products of singular values in terms of norms of wedge products, which we will use repeatedly in the proofs of our results. 

\begin{lemma}
\label{lem-Grass-3}
\noindent Let $ A$ be an $ n\times n$ real matrix with singular values $ s_{1}(A)\geq s_2(A)\geq \cdots \geq s_n(A)$.  If $ \theta_{1}, \cdots , \theta_{k}$ are unit vectors in $ \mathbb R^{n}$, then
\begin{equation}
\label{wedge-4}
\| A\theta_{1}\wedge \cdots \wedge A\theta_{k} \| \leq \sup_{\theta_1',\ldots, \theta_k' \in S^{n-1}}\| A\theta_{1}'\wedge \cdots \wedge A\theta_{k}' \| =\prod_{i=1}^{k} s_{i} (A) , 
\end{equation}
with equality if and only if  $\theta_{i}$ are orthonormal and $ {\rm span} \{ \theta_{i} , i\leq k \} $ is the subspace spanned by the eigenvectors of $AA^{\ast}$ that correspond to the largest singular values of $A.$
\end{lemma}

\begin{proof}
The inequality on the left is clear. To derive the equality, note that, for any $\theta_1,',\ldots,\theta_k'\in S^{n-1}$, 
\[
\theta_1'\wedge\cdots\wedge \theta_k' = \theta_1''\wedge \cdots \wedge \theta_k'',
\]
where
\[
\theta_j'' = \Pi_{\leq j-1}^{\perp} \theta_j',\qquad \theta_1'' = \theta_1
\]
and $\Pi_{\leq j-1}^{\perp}$ is the projection onto the orthogonal complement of the span of $\theta_1',\ldots, \theta_{j-1}'.$ This follows immediately from the fact that $a_1\wedge \cdots \wedge a_k$ is zero if $\set{a_j}$ is linearly dependent. Thus, the supremum in \eqref{wedge-4} can be taken over $\theta_1',\ldots, \theta_k'$ that are orthogonal. Over such collections, the supremum is obtained by letting $\theta_1',\ldots, \theta_k'$ be any permutation of the $k$ right singular vectors of $A$. 
\end{proof}

\noindent Next, we record in Lemma \ref{lem-Grass-2} some basic properties of this norm of wedge products that we will use. 

\begin{lemma}
\label{lem-Grass-2}
Let $x,  x_{1}, \cdots , x_{k}$ be vectors in $ \mathbb R^{n}$. Then we have the following basic properties: 
\begin{enumerate}
\item{ \bf Homogeneity}: If $\lambda_i>0$
\begin{equation}
\label{wedge-3}
 \| \lambda x_{1} \wedge \cdots \wedge  \lambda_{k} x_{k} \|  = \left(  \prod_{i=1}^{n} \lambda_{i} \right)   \|  x_{1} \wedge \cdots \wedge x_{k} \| 
\end{equation}
\item{\bf Projection formula}: Let $P_{V_{i}^{\perp}}$ be the orthogonal projection onto the orthogonal complement of $ V_{i} := {\rm span} \{ x_{1}, \cdots , x_{k} \} , V_{0} = \{0\} , 1\leq k \leq n-1$. We have
\begin{equation}
\label{wedge-1.5}
 \| x_{1} \wedge \cdots \wedge x_{k} \| = \prod_{i=1}^{k}  \| P_{V_{i-1}^{\perp}} x_{i} \|_{2}. 
\end{equation}
\item{\bf Pythagorean Theorem}: Let $e_1,\ldots, e_n$ be any orthonormal basis of $\R^n$, and define for each multi-index $I=(i_1,\ldots, i_k)$
\[e_I := e_{i_1}\wedge \cdots \wedge e_{i_k}.\]
Then, 
\begin{equation}\label{E:pyth}
    \norm{x_1\wedge\cdots\wedge x_k}^2 = \sum_{\substack{I=\lr{i_1,\ldots, i_k}\\ 1\leq i_1<\cdots<i_k\leq n}} \inprod{x_1\wedge \cdots \wedge x_k}{e_I}^2.
\end{equation}
\item {\bf Generalized Gram Identity}: Let $\Theta =\lr{\theta_1,\ldots,\theta_k}$ be an orthonormal system of $k$ vectors in $\R^n$ and write $P_\Theta$ for the orthogonal projection onto the span of the $\theta_i$. Consider arbitrary linearly independent vectors $v_1,\ldots, v_k$ in $\R^n$, and denote by $V$ the $n\x k$ matrix whose columns are $v_i$. Then
\begin{equation}\label{E:generalized-gram}
   \inprod{v_1\wedge\cdots \wedge v_k}{\theta_1\wedge\cdots\wedge \theta_k}^2= \det(P_{\Theta}VV^*P_\Theta) = \norm{P_{\Theta}v_1\wedge\cdots\wedge P_{\Theta}v_k}^2.
\end{equation}
\end{enumerate}
\end{lemma}

\begin{proof}
\noindent Homogeneity is immediate from the multi-linearity of the determinant \eqref{wedge-def}. The projection formula \eqref{wedge-1.5} follows from \eqref{wedge-def} and the fact that $\sqrt{\det(XX^*)}$ is the volume of the parallelopiped spanned by $x_1,\ldots, x_k.$. Next, the Pythagorean theorem follows from the fact that in the definition of the $\ell_2$ structure on $\Lambda^k\R^n,$ 
\[\set{e_I,\, I = \lr{i_1,\ldots,i_k},\, 1\leq i_1<\cdots< i_k\leq n}\] 
is an orthonormal basis. Finally, to show  \eqref{E:generalized-gram}, assume first that
\[\theta_j=e_j,\qquad j=1,\ldots,k\] 
are the first $k$ standard unit vectors. Then the right equality follows immediately from the Gram identity \eqref{wedge-def}. To see the left equality, write
\[v_{j}=\sum_{i=1}^n v_{j,i}e_i.\]
We have
\begin{align*}
   v_1\wedge\cdots \wedge v_k &=  \sum_{i_1,\ldots,i_k} \prod_{j=1}^k v_{j,i_j} e_{i_1}\wedge \cdots \wedge e_{i_k}= \sum_{\substack{I=\lr{i_1,\ldots,i_k}\\
   1\leq i_1<\cdots<i_k\leq n}} \lr{\sum_{\sigma \in S_k} (-1)^{\mathrm{sgn}(\sigma)} \prod_{j=1}^k v_{j,\sigma(j)}} e_I.
\end{align*}
Hence, writing $V_k$ for the matrix obtained from $V$ by keeping only the first $k$ rows, we find from the Pythagorean theorem \eqref{E:pyth}, 
\[\inprod{v_1\wedge\cdots \wedge v_k}{\theta_1\wedge\cdots\wedge \theta_k}^2 = \det(V_k)^2=\det(V_kV_k^*).\]
The case of general $\theta_i$ follows by considering any orthogonal matrix $U$ satisfying
\[\theta_i = U e_i,\qquad i = 1,\ldots,k.\]
Then
\[\inprod{v_1\wedge\cdots \wedge v_k}{\theta_1\wedge\cdots\wedge \theta_k}^2 = \det((U^T V)_k(U^TV)_k^*)=\det(P_\Theta V (P_\Theta V)^*).\]

\end{proof}

\subsection{Haar Measure on Frames}
It well-know that if $\xi$ is a standard Gaussian on $\R^n$, then $\widehat{\xi}=\xi/\norm{\xi}$ is independent of $\norm{\xi}$ and that $\widehat{\xi}$ is uniform on the unit sphere. We will need natural generalizations of these facts to orthonormal frames, Lemma \ref{L:polar} and \ref{L:haar-flags}.
\begin{lemma}[Polar Decomposition for Haar Measure on Flags]\label{L:polar}
Fix integers $n\geq k\geq 1.$ Let $\xi_1,\ldots, \xi_k\in \R^n$ be independent standard Gaussian random vectors. The following collections of random variables are independent:
\begin{equation}\label{E:norms-angles}
    \set{\norm{\xi_1},\norm{\xi_1\wedge \xi_2},\ldots, \norm{\xi_1\wedge \cdots \wedge \xi_k}},\qquad \set{\frac{\xi_1}{\norm{\xi_1}},\ldots, \frac{\xi_1\wedge \cdots \wedge \xi_k}{\norm{\xi_1\wedge \cdots \wedge \xi_k}}}.
\end{equation}
Moreover, denote by $P_{\leq i}$ the orthogonal projection onto the complement of the span of $\xi_1,\ldots, \xi_i.$ Then, the random variables terms $\norm{P_{\leq i-1} \xi_i}$ are joiltly independent.
\end{lemma}
\begin{proof}
We begin by recalling a fact from elementary probability. Namely, let $X,Y,Z$ be any random variables defined on the same probability space. Then,
\begin{equation}\label{E:fact}
    X\perp Y\qquad \text{and}\qquad X\perp Z | Y\qquad \Rightarrow\qquad X \perp (Y,Z).
\end{equation}
In words, if $X$ is independent of $Y$ and $Z$ is independent of $X$ given $Y$, then, $(Y,Z)$ is independent of $X.$  

We proceed by induction on $k.$ The case $k=1$ follows from the fact that the radial and angular parts of a standard Gaussian are independent. Suppose now $k\geq 2$ and we have proved the statement for $k-1.$ For any $\xi \in \R^k\backslash \set{0}$, let us write $P_{\xi^\perp}$ for the orthogonal projection onto the orthogonal complement to the line spanned by $\xi.$ Define for $\ell=2,\ldots,k$
\[\xi_\ell':=P_{\xi_1^\perp}\xi_\ell.\]
Note that 
\[\norm{\xi_1\wedge\cdots\wedge \xi_{\ell}}=\norm{\xi_1}\norm{\xi_2'\wedge\cdots \wedge\xi_\ell'}\]
and that
\[\xi_1\wedge\cdots\wedge \xi_{\ell}=\xi_1\wedge \xi_2' \wedge \cdots \wedge \xi_\ell'.\]
With this notation, it is enough to show that the collections
\[\set{\norm{\xi_1},\norm{\xi_2'},\ldots,\norm{\xi_2'\wedge\cdots\wedge \xi_k'}},\qquad  \set{\frac{\xi_1}{\norm{\xi_1}}, \frac{\xi_2'}{\norm{\xi_2'}},\ldots, \frac{\xi_2'\wedge\cdots\wedge \xi_k'}{\norm{\xi_2'\wedge\cdots\wedge \xi_k'}}}\]
are independent since if $X,Y$ are independent, then so are $f(X),g(Y)$ for any measurable functions $f,g$. To see this, first observe that, aside from $\norm{\xi_1}$, all other random variables in all both collections are measurable functions of $\xi_1/\norm{\xi_1},\xi_2,\ldots, \xi_k$. Hence, $\norm{\xi_1}$ is independent of all other variables in both collections. It therefore remains to check only that  
\[\mathcal A = \set{\norm{\xi_2'},\ldots,\norm{\xi_2'\wedge\cdots\wedge \xi_k'}}, \qquad \mathcal B = \set{\frac{\xi_1}{\norm{\xi_1}},\frac{\xi_2'}{\norm{\xi_2'}},\ldots, \frac{\xi_2'\wedge\cdots\wedge \xi_k'}{\norm{\xi_2'\wedge\cdots\wedge \xi_k'}}}\]
are independent. Observe that, given, $\xi_1/\norm{\xi_1}$, the random variables $\xi_\ell',\,\ell=2,\ldots,k$ are projections of iid standard Gaussians onto a fixed dimension $k$ subspace and hence are themselves iid standard Gaussians. By the inductive hypothesis, conditional on $\xi_1/\norm{\xi_1}$, the collection $\mathcal A$ is independent of the collection $\mathcal B\backslash \set{\xi_1/\norm{\xi_1}}$. Moreover, the random variables in $\mathcal A$ are independent of $\xi_1$. Invoking \eqref{E:fact} therefore completes the proof of the statement that the two collections in \eqref{E:norms-angles} are independent. To finish the proof of this Lemma, let us check that the terms
\begin{equation}\label{E:norms}
    \norm{\xi_1},\norm{P_{\leq 1}\xi_2},\ldots, \norm{P_{\leq k-1} \xi_k}
\end{equation}
are independent by induction on $k.$ When $k=1$ the claim is trivial. For the inductive step, use the projection formula \eqref{wedge-1.5} to we write
\[\norm{P_{\leq k-1} \xi_k}=\norm{\frac{\xi_1\wedge\cdots\wedge \xi_{k-1}}{\norm{\xi_1\wedge\cdots\wedge \xi_{k-1} }}\wedge \xi_k}.\]
By the first part of this Lemma, we know that 
\[\frac{\xi_1\wedge\cdots\wedge \xi_{k-1}}{\norm{\xi_1\wedge\cdots\wedge \xi_{k-1} }},\qquad \set{\norm{\xi_1\wedge \cdots \wedge \xi_j},\qquad j\leq k-1}\]
are independent. Hence, since $\xi_k$ is independent of $\xi_1,\ldots, \xi_{k-1}$, we conclude that the terms in \eqref{E:norms} are indepenedent as well.
\end{proof}
\noindent We will also use this following result:
\begin{lemma}[Haar Measure on Flags]\label{L:haar-flags}
Fix integers $n\geq k\geq 1.$ Let $\xi_1,\ldots, \xi_k\in \R^n$ be independent standard Gaussian random vectors. For each $i=1,\ldots, k$ define
\[\xi_i' = \begin{cases}\xi_1,&\quad i=1\\ P_{V_{i-1}} \xi_i,&\quad i>1\end{cases},\]
where $V_{i-1}=\mathrm{Span}\set{\xi_j,\, 1\leq j<i }^\perp$ and $P_{V_{i-1}}$ is the orthogonal projection onto $V_{i-1}$. Then $\set{\xi_i'/\norm{\xi_i'},\, i=1,\ldots, k}$ is distributed according to Haar measure on the space of such flags of orthonormal frames. In particular, if $v_1,\ldots, v_k$ is Haar-distributed in the space of $k-$frames in $\R^n$, then
\begin{equation}\label{E:gaussian-frame}
    v_1\wedge \cdots \wedge v_k \stackrel{d}{=} \frac{\xi_1\wedge \cdots \wedge \xi_k}{\norm{\xi_1\wedge \cdots \wedge \xi_k }} 
\end{equation}
\end{lemma}
\begin{proof}
The random variable $\lr{\xi_i',\, i=1,\ldots,k}$ clearly takes values in the set of $k$-frames in $\R^n.$ Moreover, it is invariant under the action of the orthogonal group on such frames and hence must be distributed according to the Haar measure. Indeed, since the angular part of a standard Gaussian is uniform on the sphere, $\xi_1'/\norm{\xi_i'}$ is uniform on $S^{n-1}$. Similarly, $\xi_2'$ is a standard Gaussian in the orthogonal complement ${\xi_1'}^\perp$ to $\xi_i'.$ Hence, $\xi_2'/\norm{\xi_2'}$ is uniform on the unit sphere in ${\xi_1'}^\perp$ and so on. Finally, to derive \eqref{E:gaussian-frame}, note that
\[
\xi_1\wedge \cdots \wedge \xi_k  = \xi_{1}'\wedge \cdots \xi_k'
\]
since the wedge product of any linearly dependent set of vectors vanishes. Combining this with the projection formula \eqref{wedge-1.5}, we conclude that
\[
\frac{\xi_1\wedge \cdots \wedge \xi_k }{\norm{\xi_1\wedge \cdots \wedge \xi_k }} = \frac{\xi_1'}{\norm{\xi_1'}}\wedge\cdots \wedge \frac{\xi_k'}{\norm{\xi_k'}}
\]
and \eqref{E:gaussian-frame} follows from the first part of this Lemma. 
\end{proof}

\section{Background on Sums of Independent Random Variables}\label{S:iid-background}
\subsection{A Result of \L ata\l a: Precise Behavior for Moments of Sums}

\noindent In the proof of our pointwise esimate Proposition \ref{P:pointwise}, we will use the following result of R. \L ata\l a \cite[Thm. 2, Cor. 2, Rmk. 2]{latala1997estimation}: 

\begin{theorem}
\label{Latala}
\noindent Let $ X_{1}, \cdots , X_{N}$ be mean zero,  independent r.v. and $ p\geq 1$. Then 
\begin{equation}
\label{Latala-1}
\left( \mathbb E\left[ \bigg| \sum_{j=1}^{N} X_{j}  \bigg|^{p}\right] \right)^{\frac{1}{p}} \simeq  \inf\left\{ t>0 : \sum_{j=1}^{N} \log \left[ \mathbb E | 1+ \frac{X_{j}}{ t}|^{p} \right] \leq p \right\}, 
\end{equation}
where $a\simeq b$ means there exist universal constants $c_1,c_2$ so that $c_1a \leq b \leq ac_2.$ Moreover, if $ X_{i}$ are also identically distributed then 
\begin{equation}
\label{Latala-2}
\left( \mathbb E \left[\bigg| \sum_{j=1}^{N} X_{j}  \bigg|^{p}\right] \right)^{\frac{1}{p}} \simeq \sup_{ \max\{ 2, \frac{ p}{ N} \} \leq s \leq p }  \frac{p}{s} \left( \frac{ N}{p}\right)^{\frac{1}{s}} \| X_{i} \|_{s}  .
\end{equation}
\end{theorem}

\subsection{Small Ball Estimates for Sums of iid Random Variables with Bounded Density}
We will have occasion to use the following standard result (e.g. Lemma 2.2 in \cite{rudelson2008littlewood}). 
    \begin{theorem}\label{T:iid-small-ball}
Let $\zeta_k,\, k=1,\ldots, n$ be iid non-negative random variables and suppose there exists $K$ such that $\mathbb P\lr{\zeta_k < \eps}\leq K \eps$ for all $\eps\geq 0$, then for all $\eps \geq 0$
    \[
    \mathbb P\lr{\sum_{k=1}^n \zeta_k^2 < n\eps^2}\leq (CK\eps)^n,
    \]
    where $C$ is an absolute constant.
\end{theorem}

\subsection{Quantitative Multivariate CLT}\label{S:MCLT}
One of our goals in this article is to measure the approximate normality for Lyapunov exponents of $X_{N,n}$ when $N\gg n$ (see Theorem \ref{T:normal}). As explained in the introduction, we will measure normality using a distance function that is a natural high-dimensional generalization of the usual Kolmogorov-Smirnov distance: 
\begin{equation}\label{E:dist-def}
d( X, Y ) := \sup_{ C \in {\cal{C}}_{k} } \left| \mathbb P ( X\in C ) -\mathbb P ( Y\in C) \right|,
\end{equation}
where $ {\cal{C}}_{k}$ is the collection of all convex subsets of $\mathbb R^{k}$. The distance function $d$ has three desirable properties. First, it is affine invariant in  sense that if $T$ is any invertible affine transformation and $A$ is any convex set, then $T^{-1}A$ is also convex and hence for any random variables $X,Y$ on the same probability space
\begin{align}
\label{E:affine-inv}d(TX,TY)&=d(X,Y).
\end{align}
The second desirable property of $d$ is that it is stable to small $\ell_2$ perturbations, as explained in Proposition \ref{Dist-lem} below. To state this result, we write
\[\mathcal N(\mu, \Sigma),\qquad \mu \in \R^k,\quad \Sigma \in \mathrm{Sym}_k^+\]
for a Gaussian with mean $\mu$ and invertible covariance $\Sigma$. 
\begin{proposition}\label{Dist-lem}
There exists $c>0$ with the following property.  Suppose that $X,Y$ are $\R^k$-valued random variables defined on the same probability space. For all $\mu\in \R^k$, invertible symmetric matrices $\Sigma\in \mathrm{Sym}_k^+$, and $\delta>0$ we have
\begin{equation}
\label{Dist-lemma-1}
d(X+Y, \mathcal N\lr{\mu, \Sigma})\leq 3 d(X,\mathcal N\lr{\mu, \Sigma})+c\delta \sqrt{\norm{\Sigma^{-1}}_{HS}}+2\mathbb P\lr{\norm{Y}_2>\delta}.
\end{equation}
\end{proposition}
We prove Proposition \ref{Dist-lem} in \S \ref{S:dist-lem-pf} below. Before doing so, however, we state the third desirable property of the distance $d$, namely, that it measures convergence in the multivariate CLT. We follow the notation in Bentkus \cite{bentkus2005lyapunov} and define
\[S:= S_{N}= X_{1} + \cdots + X_{N},\]
where $ X_{1}, \ldots, X_{N}$ are independent random vectors in $ \mathbb R^{k}$ with common mean $ \mathbb E X_{j} = 0$. We set
\[ C := {\rm cov} (S) \] 
to be the covariance matrix of $S$, which we assume is invertible. With the definition
\begin{equation}
\label{Multi-CLT-0}
 \beta_{j} := \mathbb E \| C^{-\frac{1}{2}} X_{j} \|_{2}^{3} , \ \beta:= \sum_{j=1}^{N} \beta_{j},
 \end{equation}
 we have the following \cite{bentkus2005lyapunov}:

\begin{theorem}[Multivariate CLT with Rate] 
\label{MultiCLT-theo}
\noindent There exists an absolute constant $c>0$ such that 
\begin{equation}
\label{Multi-CLT-1}
d( S, C^{\frac{1}{2}}Z ) \leq c k^{\frac{1}{4}} \beta , 
\end{equation}
where $Z\sim \mathcal N(0,\mathrm{Id}_k)$ denotes a standard Gaussian on $\R^k.$
\end{theorem}

\subsection{Proof of Proposition \ref{Dist-lem}}\label{S:dist-lem-pf}
We rely on the following result (Theorem \ref{T:Nazarov-Ball}) of Nazarov \cite{nazarov2003maximal}, which was an extension of a Theorem of Ball \cite{ball1993reverse}. \\

\noindent Let $ Z_{k}$ be the standard Gaussian in $ \mathbb R^{k}$. Let $ C$ be a positive definite symmetric matrix and let $ \gamma_{C}$ be the density of the $C^{-\frac{1}{2}}Z_{k}$, i.e. 
$$ d\gamma_{C} ( y ) := \frac{\sqrt{|\det C|}}{ (2\pi)^{\frac{k}{2}}} e^{ - \frac{ \langle C y, y\rangle   }{2}} d y . $$
Let 
\begin{equation}
\label{def-K-e}
 K_{\epsilon} := \{ x\in \mathbb R^{k} : \exists y \in K , \| x-y\|_{2} <\epsilon\}  \ {\rm and } \ K_{-\epsilon} : \{ x\in K := x+ \epsilon B_{2}^{k}  \subseteq K\} 
 \end{equation} We define 
\begin{equation}
\label{Naz-1}
\Gamma(C ) := \sup_{K\in {\cal{C}}} \left\{ \lim_{\epsilon\rightarrow 0_+ }\frac{ \gamma_{C} \left( K_{\epsilon} \setminus K\right) }{ \epsilon} \right\}. 
\end{equation}
Our proof of Proposition \ref{Dist-lem} crucially relies on the following result of Nazarov.
\begin{theorem}[\cite{nazarov2003maximal}]\label{T:Nazarov-Ball}
\noindent There exists absolute constants $ c_{1}, c_{2}>0$  such that 
\begin{equation}
\label{Naz-1}
c_{1} \sqrt{\| C\|_{HS} } \leq \Gamma(C) \leq c_{2} \sqrt{\| C\|_{HS} }.
\end{equation}
\end{theorem}

%\noindent   Note that for every $ \epsilon>0$, if $ K\in {\cal{C}}_{k}$ then also $ K_{\epsilon } $ and  $K_{-\epsilon}$  are in $ {\cal{C}}_{k}$. It is a standard fact (see for example the proof of proposition 2.1 in \cite{ledoux2001concentration}) that \eqref{Naz-1} implies that for every $ K\in {\cal{C}}_{k}$ and $\epsilon>0$, 
%\begin{equation}
%\label{Naz-2}
%\left| \mathbb P ( C^{-\frac{1}{2}} Z_{k} \in K_{\epsilon} ) - \mathbb P ( C^{-\frac{1}{2}} Z_{k} \in K_{ -\epsilon} ) \right|  \leq c_{2} \epsilon \sqrt{\| C\|_{HS} }.\\
%\end{equation}

\noindent We will need an elementary corollary of \eqref{Naz-1}.
\begin{corollary}\label{C:Naz}
For any $\eps\in (0,1)$ and $K\in \mathcal D_k$, we have
\begin{equation}
\label{Naz-2}
\left| \mathbb P ( C^{-\frac{1}{2}} Z_{k} \in K_{\epsilon} ) - \mathbb P ( C^{-\frac{1}{2}} Z_{k} \in K_{ -\epsilon} ) \right|  \leq c_{2} \epsilon \sqrt{\| C\|_{HS} }.\\
\end{equation}
\end{corollary}
\begin{proof}
Our argument is along the lines of the proof in \cite{bentkus2005lyapunov} of equations (1.3), (1.4) or \cite{bentkus2003dependence} equation (1.2), which obtain similar statements in the special case where $C=\mathrm{Id}$. We will use a standard estimate (see e.g. Lemma 10.5 in \cite{kallenberg2002foundations}) that if $K$ is any convex set in $\R^k$ and $(\partial K)_\epsilon$ is the $\epsilon$-neighborhood of the boundary of $K$, then 
\begin{equation}\label{E:Kal}
    \mathrm{vol}((\partial K)_\epsilon) \leq 2\left[\lr{1+\frac{\epsilon}{r(K)}}^k-1\right]\mathrm{vol}(B_{r(K)}),
\end{equation}
where $r(K)$ is the radius of the smallest ball containing $K$, $\mathrm{vol}$ is the Euclidean volume, and $B_{r}$ is a ball of radius $r$. The crucial point is that the right hand side is bounded for all $\epsilon\in (0,1)$ by $c\epsilon$ with $c$ depending only on $r(K)$ and the ambient dimension $k$.

To derive from this estimate \eqref{Naz-2}, observe that for every $ \epsilon>0$, if $ K\in {\cal{C}}_{k}$ then also $ K_{\epsilon } $ and  $K_{-\epsilon}$ are in $ {\cal{C}}_{k}$. Note also that the difference of probabilities on the left-hand side of \eqref{Naz-2} can be bounded above by  $ \mathbb P (W\in K_{\epsilon} \setminus K )+ \mathbb P (W\in K\setminus  K_{-\epsilon}  ) $ where $ W:= C^{-\frac{1}{2}} Z_{k} $. It is therefore enough to check that each of these probabilities is in turn bounded above by the right hand side of \eqref{Naz-2}. To see this for $\mathbb P (W\in K_{\epsilon} \setminus K )$, denote for $K$ convex and $t>0$,
\[
 \omega_{K}(t) :=  \mathbb P \left( W \in K_{t} \setminus K\right).
\]
Since  $ K_{t+ \epsilon} \setminus K = ( K_{t+ \epsilon} \setminus K_{t}) \cup (K_{t} \setminus K) $ we have that 
 \begin{equation}\label{E:semi-group}
     \omega_{K} (t+ \epsilon) - \omega_{K} (t) = \omega_{K_{t} } (\epsilon). 
 \end{equation}
This relation, together with \eqref{E:Kal}, implies that $\omega_K(t)$ is absolutely continuous. Indeed, we may write
\begin{align*}
    \omega_{K_t}(\epsilon) =  \mathbb P\lr{Z_k\in C^{1/2}K_{t+\epsilon}\backslash C^{1/2}K_t} \leq \frac{\mathrm{vol}\lr{C^{1/2}K_{t+\epsilon}\backslash C^{1/2} K_t}}{(2\pi)^{k/2}} .
\end{align*}
Denoting by $\lambda_{\text{max}}$ the maximal eigenvalue of $C^{1/2}$, we may write
\[
C^{1/2}K_{t+\epsilon} = C^{1/2}\lr{K_{t}}_\epsilon\subseteq (C^{1/2}K_{t})_{\lambda_{\text{max}}\epsilon}.
\]
Thus, 
\begin{align*}
    \omega_{K_t}(\epsilon) \leq \frac{\mathrm{vol}\lr{(C^{1/2}K_t)_{\lambda_{\text{max}}\epsilon}\backslash C^{1/2}K_t}}{(2\pi)^{k/2}} .
\end{align*}
Denoting by $R$ the radius of the smallest ball containing $(C^{1/2}K_1)_{\lambda_{\text{max}}}$, we obtain from \eqref{E:Kal} that there exists a constant $c>0$ depending on $C$, $k$ and $R$ so that for any $0\leq t,\epsilon \leq 1$
\[
\omega_{K_t}(\epsilon) \leq c\epsilon.
\]
Thus, using \eqref{E:semi-group}, we indeed see that $\omega_K(t)$ is absolutely continuous on $[0,1]$. Hence, its derivative $\omega_K'(t)$ exists almost everywhere and
\[
\omega_K(t)=\int_0^t \omega_K'(s)ds.
\]
Moreover, combining \eqref{E:semi-group} with \eqref{Naz-1}, we find for any $t\in (0,1)$ that
\[
0\leq \omega_K'(t)= \lim_{\epsilon\gives 0_+}\frac{\omega_{K} (t+ \epsilon) - \omega_{K} (t) }{\epsilon} = \lim_{\epsilon\gives 0_+}\frac{\omega_{K_t}(\epsilon)}{\epsilon}\leq c\sqrt{||C||
_{HS}}
\]
Hence, since $\omega_K(0)=0$, we find that $ \omega_{K} (t) \leq  ct\sqrt{\| C\|_{HS}}  $, for all $t\leq 1$. Using a similar argument for $\bar{\omega}_{K}(t) := \mathbb P (W\in K\setminus  K_{- t}  )$, we conclude that \eqref{Naz-2} indeed holds. 

\end{proof}

%We do this by first reducing to a statement about the standard Gaussian $Z_k$ and then quoting known results. For this reduction, note that the left hand side of \eqref{Naz-2} is equal to 
%\[
%\int_{\R^k} {\bf 1}_{\set{\norm{C^{-1/2}x- \partial K}\leq \eps}} d\gamma_{\mathrm{I}}(x),
%\]
%where $\partial K$ is the boundary of $K.$ For any $x\in \R^k$, we have
%\[
%\norm{C^{-1/2}x- \partial K}=\frac{\norm{C^{1/2}}_{HS}}{\norm{C^{1/2}}_{HS}}\norm{C^{-1/2}x- \partial K}\geq \norm{C^{1/2}}_{HS}^{-1} \norm{x- \partial \left[C^{1/2}K\right]}.
%\]
%Thus, we find that the left hand side of \eqref{Naz-2} is bounded above by 
%\%[
%\int_{\R^k} {\bf 1}_{\set{\norm{x- \partial \left[C^{1/2} K\right]}\leq \eps\norm{C^{1/2}}_{HS}}} d\gamma_{\mathrm{I}}(x) = \left| \mathbb P (  Z_{k} \in (C^{1/2}K)_{\epsilon_C} ) - \mathbb P ( Z_{k} \in (C^{1/2} K)_{ -\epsilon_C} ) \right|  ,
%\]
%where $\eps_C = \eps \norm{C^{1/2}}_{HS}.$ Finally, 

We are now ready to prove Proposition \ref{Dist-lem}. Note that if $S,T$ are any events on the same probability space, then
\[\abs{\mathbb P\lr{S}-\mathbb P\lr{T}}\leq \mathbb P\lr{S\Delta T},\]
where $S\Delta T$ is the symmetric difference. For any convex set $A\subseteq \R^k$, we have
\[\abs{\mathbb P\lr{X+Y\in A}- \mathbb P\lr{X\in A}}\]
is bounded above by 
\[\abs{\mathbb P\lr{X+Y\in A,\, \norm{Y}_2\leq \delta}- \mathbb P\lr{X\in A,\, \norm{Y}_2\leq \delta}}+2\mathbb P\lr{\norm{Y}_2>\delta}.\]
Note that 
\[\set{X+Y\in A,\, \norm{Y}_2\leq \delta}\Delta\set{X\in A,\, \norm{Y}_2\leq \delta}~\subseteq~\{ X\in A_{\delta}\setminus A_{-\delta}\}.\]
Thus, we find
\begin{align*}
    \abs{\mathbb P\lr{X+Y\in A}- \mathbb P\lr{X\in A}}&\leq \mathbb P\lr{X\in   A_\delta} -  \mathbb P\lr{X\in   A_{-\delta}} + 2 \mathbb P\lr{\norm{Y}_2>\delta}\\
    &\leq
 |  \mathbb P\lr{X\in   A_\delta} -  \mathbb P\lr{ \mathcal N\lr{\mu, \Sigma}\in   A_\delta} |\\
 &+ |  \mathbb P\lr{X\in   A_{-\delta}}  - \mathbb P\lr{ \mathcal N\lr{\mu, \Sigma}\in   A_{-\delta}} |\\
 &+  \mathbb P\lr{ \mathcal N\lr{\mu, \Sigma}\in   A_\delta} -  \mathbb P\lr{ \mathcal N\lr{\mu, \Sigma}\in   A_{-\delta}}  + 2 \mathbb P\lr{\norm{Y}_2>\delta}\\
 &\leq 
 2d(X, \mathcal N\lr{\mu, \Sigma})  + c_{0}\delta \sqrt{\norm{\Sigma^{-1}}_{HS}} + 2\mathbb P\lr{\norm{Y}_2>\delta}
\end{align*}
where we have used \eqref{Naz-2}. 
Putting this all together, we find
\begin{align*}
    d(X+Y, \mathcal N\lr{\mu, \Sigma}) &\leq d(X, \mathcal N\lr{\mu, \Sigma}) + d(X+Y, X)\\
    &\leq 3d(X, \mathcal N\lr{\mu, \Sigma}) + c_{0}\delta \sqrt{\norm{\Sigma^{-1}}_{HS}} + 2\mathbb P\lr{\norm{Y}_2>\delta}.
\end{align*} \hfill $\square$

\section{Roadmap for Proofs of Theorems \ref{T:global} and \ref{T:normal}}\label{S:roadmap}
In this section, we explain the organization of the proofs of Theorems \ref{T:global} and \ref{T:normal}.
Our starting point is in \S \ref{S:small-ball-use}. There, in  Proposition \ref{P:pointwise-small-ball} we explain how to provide surprisingly useful bounds on the size of the difference
\begin{equation}\label{E:local-global-diff}
    \frac{1}{N}\log \norm{X_{N,n}(\Theta)}-\lr{\lambda_1+\cdots + \lambda_k}
\end{equation}
using small ball estimates on determinants of volumes of random projections. This makes precise \eqref{E:pointwise-informal}. We remind the reader that $\lambda_1,\ldots,\lambda_k$ are the top $k$ Lyapunov exponents for $X_{N,n}$ and that 
\[X_{N,n}(\Theta) = X_{N,n}\theta_1\wedge\cdots X_{N,n}\theta_k,\]
where $\theta_j$ are is a fixed orthonormal $k-$frame in $\R^n.$ We think of $X_{N,n}(\Theta)$ as a pointwise analog of $\lambda_1+\cdots+ \lambda_k$ since by Lemma \ref{lem-Grass-3} the supremum over $\Theta$ of $\frac{1}{N}\log\norm{X_{N,n}(\Theta)}$ equals $\lambda_1+\cdots+ \lambda_k.$

Using Proposition \ref{P:pointwise-small-ball}, we analyze in \S \ref{S:concentration} the concentration properties of $\frac{1}{N}\log\norm{X_{N,n}(\Theta)}.$ By Lemma \ref{L:pointwise-iid} it is a sum of independent random variables, allowing us to apply Theorem \ref{Latala} several times. The main result is Proposition \ref{P:pointwise}, whose proof is the most technical part of this article. 

Combining these concentration estimates for $\frac{1}{N}\log\norm{X_{N,n}(\Theta)}$ with the bounds on \eqref{E:local-global-diff} derived in Proposition \ref{P:pointwise-small-ball}, we derive Theorem \ref{T:global} in \S \ref{S:global-pf}, giving quantitative estimates about convergence of the global distribution of singular values of $X_{N,n}$ to the Triangle Law. 

Finally, in \S \ref{S:normal-outline}, we combine Theorem \ref{T:sums} with Proposition \ref{P:pointwise} and the multivariate CLT (see \S \ref{S:MCLT}) to prove the approximate normality of Lyapunov exponents stated in Theorem \ref{T:normal}.

\section{Lyapunov Sums Via Small Ball Estimates}\label{S:small-ball-use}
The purpose of this section is to explain how to use small ball estimates on volumes of random projections to obtain  concentration estimates on the difference
\[\frac{1}{N}\log \norm{X_{N,n}(\Theta)}-\sup_{\Theta'\in \mathrm{Fr}_{n,k}}\frac{1}{N}\log \norm{X_{N,n}(\Theta')}=\frac{1}{N}\log \norm{X_{N,n}(\Theta)}-\lr{\lambda_1+\cdots + \lambda_k}\]
between the sum of the first $k$ Lyapunov exponents of $X_{N,n}$ and the ``pointwise'' analog of this quantity evaluated at any fixed orthonormal system $\Theta=\lr{\theta_1,\dots,\theta_k}$ of $k$ vectors in $\R^n.$ Our main result is the following
\begin{proposition}\label{P:pointwise-small-ball}
There exists $C>0$ with the following property. For any $\eps\in (0,1)$ and any $\Theta \in \mathrm{Fr}_{n,k}$ we have
 \[\mathbb P\lr{\abs{\frac{1}{N}\log \norm{X_{N,n}(\Theta)} -\sup_{\Theta' \in \mathrm{Fr}_{n,k}}\frac{1}{N}\log \norm{X_{N,n}(\Theta')}}\geq \frac{k}{2N}\log\lr{\frac{n}{k\eps^2}}}\leq \lr{C\eps}^{k/2}.\]
\end{proposition}
\begin{proof}
The key observation is the following:
\begin{lemma}\label{L:small-ball-bounds}
For any $\eps\in (0,1)$ and any $\Theta \in \mathrm{Fr}_{n,k}$ we have
 \[\mathbb P\lr{\abs{\frac{1}{N}\log \norm{X_{N,n}(\Theta)} -\sup_{\Theta' \in \mathrm{Fr}_{n,k}}\frac{1}{N}\log \norm{X_{N,n}(\Theta')}}\geq \frac{1}{N}\log\lr{\frac{1}{\eps}}}\leq \mathbb P\lr{\norm{P_F(\Theta)}\leq \eps},\]
 where $F$ is a Haar distributed $k-$dimensional subspace of $\R^n$ and 
 \[P_F(\Theta)=P_F\theta_1\wedge \cdots\wedge P_F\theta_k\] 
 with $P_F$ denoting the orthogonal onto $F.$
\end{lemma}
\begin{proof}[Proof of Lemma \ref{L:small-ball-bounds}]
Denote by $v^{(1)},\ldots, v^{(n)}$ the right singular vectors of $X_{N,n}$ corresponding to its singular values $s_1\geq \cdots \geq s_n$. By abuse of notation, we will write $X_{N,n}:\Lambda^k\R^n\gives \Lambda^k\R^n$ for the linear transformation given by 
\[X_{N,n}(x_1\wedge\cdots\wedge x_k)=X_{N,n}x_1\wedge\cdots \wedge X_{N,n}x_k.\]
The right singular vectors of $X_{N,n}$ acting on $\Lambda^k\R^n$ are 
\[v^{(I)}:=v^{(i_1)}\wedge\cdots \wedge v^{(i_k)},\qquad I = \lr{i_1,\ldots, i_k},\, i_1<\cdots <i_k\]
and the corresponding singular values are
\[s_I:=\prod_{i\in I}s_i.\]
%Note that, for any orthonormal system $v_1,\ldots, v_k$ of vectors in $\R^n$ and any $\theta_1,\ldots, \theta_k\in \R^n,$ the Gram Identity \eqref{E:gram} shows
%\[\inprod{v_1\wedge \cdots \wedge v_k}{\theta_1\wedge \cdots \wedge \theta_k}^2= \norm{P_V\theta_1\wedge \cdots \wedge P_V \theta_k}^2,\]
%where $P_V$ is the orthogonal projection onto the span of the $v_1,\ldots, v_k.$
Hence, the Pythagorean theorem for wedge products and the generalized Gram identity (see Lemma \ref{lem-Grass-2}) yield
\begin{align*}
   \norm{X_{N,n}(\Theta)}^2 &= \sum_{\substack{I=\lr{i_1,\ldots,i_k}\\ 1\leq i_1<\cdots < i_k \leq n}} s_I^2 \inprod{v^{(I)}}{\Theta}^2\geq \lr{\prod_{i=1}^k s_i^2} \inprod{v^{(1,\ldots,k)}}{\Theta}^2 = \lr{\prod_{i=1}^k s_i^2}\norm{P_k(\Theta)}^2,
\end{align*}
where in the last equality we've denote by $P_k$ the orthogonal projection into the span of the top $k$ right singular vectors of $X_{N,n}$. We therefore obtain, using Lemma \ref{lem-Grass-3}: 
\begin{align*}
    0&\geq \frac{1}{N}\log \norm{X_{N,n}(\Theta)} -\sup_{\Theta' \in \mathrm{Fr}_{n,k}}\frac{1}{N}\log \norm{X_{N,n}\Theta'}\\
    &= \frac{1}{2N}\log\lr{\frac{\norm{X_{N,n}(\Theta)}^2}{\prod_{i=1}^k s_i^2}}\\
    &=\frac{1}{N}\log\norm{P_k(\Theta)}.
\end{align*}
Since $X_{N,n}$ is invariant under right multiplication by a Haar orthogonal matrix, we see that $P_k$ is equal in distribution to the orthogonal projection onto a Haar distributed $k$-dimensional subspace of $\R^n.$
\end{proof}

\noindent In order to apply Lemma \ref{L:small-ball-bounds} we need small ball estimates on $\norm{P_F(\Theta)}$. Gaussian analogs of such estimates are essentially available in the literature, but are phrased in the language of determinants of random matrices. To reduce to this case, note that if $F$ is Haar distributed among $k-$dimensional subspaces of $\R^n$, an orthonormal basis $v_1,\ldots,v_k$ for $F$ is Haar distributed on the space of such $k-$frames in $\R^n$. Thus, by \eqref{E:gaussian-frame} from Lemma \ref{L:haar-flags}, we find
 \begin{align}
\notag \norm{P_F(\Theta)} &= 
 \| P_{F}  \theta_{1} \wedge \cdots \wedge P_{F} \theta_{k} \|\\
 \notag &= \abs{\inprod{v_1\wedge\cdots\wedge v_k}{\theta_1\wedge\cdots\wedge \theta_k}}\\
 \notag &\stackrel{d}{=} \abs{\inprod{\frac{g_1\wedge \cdots\wedge g_k}{\norm{g_1\wedge \cdots \wedge g_k}}}{\theta_1\wedge\cdots\wedge \theta_k}}\\
 \notag&\stackrel{d}{=}  \frac{  \| G \theta_{1} \wedge \cdots \wedge G \theta_{k} \| }{  {\rm det} ( G G^{\ast} )^{\frac{1}{2}}  }\\
 \label{E:wedge-det}&=\lr{\frac{  \det(G_kG_k^*) }{  {\rm det} ( G G^{\ast} )  }}^{\frac{1}{2}}
 \end{align}
where $\stackrel{d}{=}$ denotes equality in distribution, $G$ is a $k\x n$ matrix with iid standard Gaussian columns $g_i$, $G_k$ is the obtained from $G$ by keeping only the first $k$ columns, and we have used the Gram identity \eqref{wedge-def} and \eqref{E:generalized-gram} in the last two lines. The relation \eqref{E:wedge-det}, combined in Lemma \ref{L:small-ball-bounds}, therefore gives that
 \[
 \mathbb P\lr{\abs{\frac{1}{N}\log \norm{X_{N,n}(\Theta)} -\sup_{\Theta' \in \mathrm{Fr}_{n,k}}\frac{1}{N}\log \norm{X_{N,n}(\Theta')}}\geq \frac{1}{N}\log\lr{\frac{1}{\eps}}}
 \]
 is bounded above by 
 \begin{equation}\label{E:small-ball-det}
   \mathbb P\lr{\lr{\frac{  \det(G_kG_k^*) }{  {\rm det} ( G G^{\ast} )  }}^{\frac{1}{2}}\leq \eps},   
 \end{equation}
To complete the proof, we recall the following result.
\begin{proposition}[Lemma 4.2 in \cite{paouris2013small}]
\label{Gaussian-matrix-det-est}
\noindent There exist universal constants $c,C>0$ with the following property. Let $ G$ be a $k\times n $ matrix with iid standard $\mathcal N(0,1)$ Gaussian entries. Then 
\begin{equation}
\label{G-M-D-E-1}
\left( \mathbb E \left[{\det} ( G G^{\ast} )^{\frac{p}{2k}}\right]\right)^{\frac{1}{p}} \leq C \sqrt{n} , \ 0< p \leq k n .
\end{equation}
and 
\begin{equation}
\label{G-M-D-E-2}
\left( \mathbb E \left[{\det} ( G G^{\ast} )^{-\frac{p}{2k}}\right]\right)^{-\frac{1}{p}} \geq c \sqrt{n} , \ 0< p \leq k (n-k+1- e^{-k(n-k+1)}). 
\end{equation}
\end{proposition}

\noindent This allows use to estimate the probability in \eqref{E:small-ball-det} via the following corollary, which in combination with $\eqref{E:small-ball-det}$ and Lemma \ref{L:small-ball-bounds} completes the proof of Proposition \ref{P:pointwise-small-ball}.

\begin{corollary}\label{C:small-ball}
\noindent There exists a universal constant $c>0$ with the following property. Let $G$ be a $k\x n$ matrix with iid Standard Gaussian entries, and denote by $G_k$ the matrix obtained from $G$ by keeping only the first $k$ columns. 
%Let $ \theta_{1}, \cdots, \theta_{k}$  be orthonormal vectors vectors in $ \mathbb R^{n}$, $ F$ be a $k$-dimensional subspace of $\mathbb R^n$ distributed according to the Haar (uniform) measure.
%and $G$ be a $k\times n $ matrix with iid standard Gaussian entries. Then 
\begin{equation}
\label{invar-cor-2}
\mathbb P\left(  \lr{\frac{\det(G_kG_k^*)}{\det(GG^*)}}^{\frac{1}{2k}}  \leq \varepsilon \sqrt{ \frac{k}{n}}\right) \leq \left(  c \varepsilon \right)^{\frac{k}{2}}   , \varepsilon >0 
%\mathbb P\left(  \| P_{F}  \theta_{1} \wedge \cdots \wedge P_{F} \theta_{k} \|^{\frac{1}{k}}  \leq \varepsilon \sqrt{ \frac{k}{n}}\right) \leq \left(  c \varepsilon \right)^{\frac{k}{2}}   , \varepsilon >0 
 \end{equation}
\end{corollary}
\begin{proof}
Using \eqref{G-M-D-E-1} and \eqref{G-M-D-E-2} and Markov's inequality, we have that for $t\geq 1$, 
$$ \mathbb P \left(  {\det} ( G G^{\ast} )^{\frac{1}{2k}} \geq t C \sqrt{n} \right) \leq \frac{1}{ t^{nk}} \ {\rm and } \  \mathbb P \left(  {\det} ( G G^{\ast} )^{\frac{1}{2k}} \leq \varepsilon c  \sqrt{n} \right) \leq (c \varepsilon)^{k(n-k+1-e^{-k(n-k+1)}) }  .$$
Note that $ G_{k}$ has the same distribution as a $k\times k$ matrix with iid standard $\mathcal N(0,1)$ Gaussian entries. So, we have that 
\begin{align*}
   \mathbb P\left(  \lr{\frac{\det(G_kG_k^*)}{\det(GG^*)}}^{\frac{1}{2k}}    \leq \varepsilon \sqrt{\frac{k}{n}}  \right) &\leq  \mathbb P \left( \frac{ {\rm det} ( G_{k} G_{k}^{\ast} )^{\frac{1}{2k}}  }{  {\rm det} ( G G^{\ast} )^{\frac{1}{2k}}  }  \leq \varepsilon \sqrt{\frac{k}{n}} \ {\rm and }   \ {\rm det} ( G G^{\ast} )^{\frac{1}{2k}}  \leq t C\sqrt{n}  \right)\\ &+ \mathbb P \left(   {\rm det} ( G G^{\ast} )^{\frac{1}{2k}}  \geq t C\sqrt{n} \right) \\
    &\leq \mathbb P \left(  {\rm det} ( G_{k} G_{k}^{\ast} )^{\frac{1}{2k}}   \leq \varepsilon t C\sqrt{k}  \right)  + \mathbb P \left(   {\rm det} ( G G^{\ast} )^{\frac{1}{2k}}  \geq t C\sqrt{n} \right)\\
    &\leq (c^{\prime}t\varepsilon)^{k(1-e^{-k}) } +  \frac{1}{ t^{kn}} \leq \left(  c \varepsilon \right)^{k/2},
\end{align*}
where in the last line we've taken $ t=  \varepsilon^{-1/2n}$. 
\end{proof}

\end{proof}

%let $M$ be any $k\times n$ matrix and $ K $ be a compact subset of $\mathbb R^{n}$ that has positive $m$-dimensional volume for $ m\geq k $. Suppose that $ U $ is uniformly distributed on the orthogonal group and $F$ uniformly distributed on the Grassmanian $G_{n,k}$ of $k$-planes in $\R^n$. Writing $P_V$ for the orthogonal projection onto a subspace $V$, the following equality in distribution
%\begin{equation}
%\label{invar-0}
%GU = G \ {\rm in } \ {\rm distribution} 
% \end{equation}
%\begin{equation}
%\label{invar}
%\vol_k\lr{ G U K}  = {\rm det}  ( G G^{\ast} )^{\frac{1}{2}} \vol_k\lr{ P_{U^{\ast} H } K }
%\end{equation}
%is well-known (cf e.g. the Proof of Proposition 4.1 in \cite{paouris2013small} or in \cite{paouris2014central,vitale2008gaussian}), where we've set $ H:= {\rm Range}( G^{\ast})= {\rm span } \{ g_{1}, \cdots , g_{k} \} $, with  $g_{i}$ denoting the rows of $ G$. Put another way, in distribution, 
%\begin{equation}
%\label{invariant}
% \vol_k\lr{ P_{F}K } = \frac{ \vol_k(G K) }{ {\rm det}  ( G G^{\ast} )^{\frac{1}{2}}  } .
%\end{equation}
%Note that by the rotational invariance of the Haar measure, we may assume that $ v_{i}$ is the standard basis. Thus, by \eqref{wedge-def},
%\[\norm{P_Fv_1\wedge \ldots\wedge P_Fv_k}\]
%is the volume of the unit cube in in the first $k$ coordinates of $ \R^n$ projected onto $F.$ Thus, \eqref{invar-cor-1} follows immediately from \eqref{invariant}.

\section{Concentration for $\frac{1}{N}\log\norm{X_{N,n}(\Theta)}$}\label{S:concentration}
As mentioned above, a key step towards proving Theorems \ref{T:global} and \ref{T:normal} is to obtain precise concentration estimates for 
\begin{equation}
    \label{E:XnN-theta-def}\frac{1}{N}\log\norm{X_{N,n}(\Theta)}=\frac{1}{N}\log\norm{X_{N,n}\theta_1\wedge\cdots\wedge X_{N,n}\theta_k},
\end{equation}
where $\Theta = \lr{\theta_1,\ldots,\theta_k}$ is a fixed orthonormal system in $\R^n.$ Define
\begin{equation}\label{E:M-xi-def}
    M_j:=n-j+1,\qquad\xi_{n,k}=\frac{1}{n}\sum_{j=1}^k \frac{1}{M_j}
\end{equation}
and as in \eqref{E:mu-def} set
\[\mu_{n,j}:=\frac{1}{2}\E{\log\lr{\frac{1}{n}\chi_{n-j+1}^2}}.\]
Our main result about the concentration for $\log\norm{X_{N,n}(\Theta)}$ is the following.
\begin{proposition}\label{P:pointwise}
There exists a universal constant $c>0$ with the following property. Fix any orthonormal system $\Theta$ of $k$ vectors in $\R^n$. With $X_{N,n}(\Theta)$ as in \eqref{E:XnN-theta-def}, we have

\begin{equation}
\label{proposition9-1}
\mathbb P\left( \abs{ \frac{ 1 }{ nN }\log \norm{X_{N,n}(\Theta)} - \frac{1}{n}\sum_{j=1}^k \mu_{n,j}}\geq s \right) \leq 2 \exp\left\{ - c n N \min\{ M_{k} s , \frac{s^{2}}{ \xi_{n,k}}\} \right\}, \qquad s>0 . 
\end{equation}
\end{proposition}
\begin{remark}
The double behavior in the exponent of the estimates \eqref{proposition9-1} is of Bernstein-type. We do not use any off-the-shelf Bernstein estimates for deriving it however, relying instead on Theorem \ref{Latala} of \L ata\l a \cite{latala1997estimation}. One advantage of our approach is that \L ata\l a's estimates are all reversible (i.e. have matching upper and lower bounds). Hence, with a bit more work it is possible to show that the estimate \eqref{proposition9-1} is sharp. We will not need this fact, however, and will provide only a proof of the upper bound.
\end{remark}
\begin{remark}
Although we focus in this article on the Gaussian case, we believe it is possible to prove that Proposition \ref{P:pointwise} holds under minimal assumptions on the distribution of the entries of $A_i.$ Somewhat weaker results in this directions are proved in \cite[Thms. 7,8]{le1982theoremes} and \cite[Thm. 5.1]{bougerol1985concentration}.
\end{remark}
\begin{remark}
By Lemma \ref{L:pointwise-iid} below, we have
\[ \E{\frac{ 1 }{ N }\log \norm{X_{N,n}(\Theta)}} = \sum_{j=1}^k \mu_{n,j}.\]
\end{remark}
The proof of Proposition \ref{P:pointwise} proceeds from the observation that for the Gaussian case we consider here, $\log X_{N,n}(\Theta)$ is a sum of independent random variables. 
\begin{lemma}\label{L:pointwise-iid}
Fix $n,N\geq 1$ and $1\leq k \leq n$ as well as a collection $\Theta$ of $k$ orthonormal vectors in $\R^n.$ We have
\[\log\norm{X_{N,n}(\Theta)}\stackrel{d}{=}\sum_{i=1}^N \sum_{j=1}^k Y_{i,j},\]
where $\stackrel{d}{=}$ denotes equality in distribution, $Y_{i,j}$ are independent, and for each $i,j$ the random variable $Y_{i,j}$ is distributed like the logarithm  $\frac{1}{2}\log(\frac{1}{n}\chi_{n-j+1}^2)$ of a normalized chi-squared random variable with $n-j+1$ degrees of freedom. 
\end{lemma}
\begin{proof}
We have
\begin{align}
\notag    \log\norm{X_{N,n}(\Theta)} &= \log \norm{A_N\cdots A_1\lr{\Theta}}\\
\label{E:one-term}    &= \log \norm{A_N\cdots A_2 \frac{A_1\lr{\Theta}}{\norm{A_1(\Theta)}}}+\log\norm{A_1(\Theta)}.
\end{align}
Note that by Lemma \ref{L:polar}, we have that
\[\norm{A_1(\Theta)} = \norm{A_1\theta_1\wedge \cdots \wedge A_1 \theta_k}\]
is independent of 
\[\frac{A_1(\Theta)}{\norm{A_1(\Theta)}}=\frac{A_1\theta_1\wedge \cdots \wedge A_1 \theta_k}{\norm{A_1\theta_1\wedge \cdots \wedge A_1 \theta_k}}.\]
Hence the two terms in \eqref{E:one-term} are independent. Moreover, $A_2\lr{\frac{A_1(\Theta)}{\norm{A_1(\Theta)}}}$ is independent of $A_3,\ldots, A_N$ and, in distribution, we have
\begin{equation}\label{E:dist-wedge}
    A_2\lr{\frac{A_1(\Theta)}{\norm{A_1(\Theta)}}}\stackrel{d}{=}A_2(\Theta).
\end{equation}
Indeed, we may write
\[\frac{A_1(\Theta)}{\norm{A_1(\Theta)}}=\frac{A_1\theta_1\wedge \cdots \wedge A_k\theta_k}{\norm{A_1\theta_1\wedge \cdots \wedge A_k\theta_k}}= \frac{A_1\theta_1}{\norm{A_1\theta_1}}\wedge \frac{\Pi_{\leq 1}A_1\theta_1}{\norm{\Pi_{\leq 1}A_1\theta_1}}\wedge \cdots \wedge \frac{\Pi_{\leq k-1} A_1\theta_k}{\norm{\Pi_{\leq k-1} A_1\theta_k}},\]
where we've written $\Pi_{\leq i}$ for the projection onto the complement of the span of $A_1\theta_1,\ldots,A_1\theta_i.$ Next, we may choose an orthogonal matrix $M$ so that
\[\Pi_{\leq i-1}A_1\theta_i = M e_i, \]
where $e_i$ is the $i^{th}$ standard basis vector. For this choice of $M,$ we find
\[\frac{A_1(\Theta)}{\norm{A_1(\Theta)}}=Me_1\wedge \cdots \wedge Me_k= M(e_1\wedge \cdots \wedge e_k).\]
Since $A_2 \stackrel{d}{=}A_2M$, we conclude that \eqref{E:dist-wedge} holds. In particular, we find that, in distribution, 
\[\log \norm{X_{N,n}(\Theta)}\stackrel{d}{=}\sum_{i=1}^N \log \norm{A_i(\Theta)}\]
is a sum of iid terms. Finally, for any fixed $i=1,\ldots, N$ \[\norm{A_i(\Theta)}\stackrel{d}{=}\norm{\xi_1\wedge \cdots \wedge \xi_k},\]
where $\xi_i$ are iid $n$-dimensional standard Gaussians. Hence, by the projection formula \eqref{wedge-1.5}, we find that
\[\norm{A_i(\Theta)}\stackrel{d}{=}\prod_{i=1}^k \norm{P_{\leq i-1} \xi_i},\]
where we've denoted by $P_{\leq j}$ the projection onto the orthogonal complement of the span of $\xi_1,\ldots,\xi_j$. The terms in the product are independent by Lemma \ref{L:polar}, and the distribution of the $i^{th}$ term is precisely the same as that of $\sqrt{\frac{1}{n}\chi_{n-i+1}^2}$, completing the proof. 
\end{proof}

Lemma \ref{L:pointwise-iid} allows us to obtain precise estimates on the rate of growth of moments of $\log X_{N,n}(\Theta)$ using the result of \L ata\l a \cite{latala1997estimation} (Theorem \ref{Latala} above). These moment estimates, in turn, yield Proposition \ref{P:pointwise} via Markov's inequality applied to the optimal power of $\log \norm{X_{N,n}(\Theta)}$. We carry out these details in \S \ref{S:pointwise-pf}.

\subsection{Details for Proof of Proposition \ref{P:pointwise}}\label{S:pointwise-pf}
The purpose of this section is to prove Proposition \ref{P:pointwise}. Throughout this section $C,C^{\prime},c,c^{\prime}$ etc will be universal constants that may change from line to line. Recalling from \eqref{E:M-xi-def} the notation
\begin{equation}\label{E:M-xi-mu-def}
M_j=n-j+1,\quad \xi_{n,k}=\frac{1}{n}\sum_{j=1}^k\frac{1}{M_j},
\end{equation}
we seek to show that for $s>0$ 
\begin{equation}\label{E:pointwise-goal-1}
\mathbb P\lr{\abs{\frac{1}{nN}\log\norm{X_{N,n}(\Theta)}- \frac{1}{n}\sum_{j=1}^k\mu_{n,j}}\geq s}\leq 2 \exp\lr{-cnN\min\set{M_ks,\frac{s^2}{\xi_{n,k}}}}.
\end{equation}
where we remind the reader that as in \eqref{E:mu-def}, we've set
\[\mu_{n,j}:=\frac{1}{2}\E{\log\lr{\frac{1}{n}\chi_{n-j+1}^2}}.\]
According to Lemma \ref{L:pointwise-iid}, we have in distribution that
\[\frac{2}{N}\log\norm{X_{N,n}(\Theta)} =\frac{1}{N}\sum_{i=1}^N T_i,\]
where $T_i$ are independent and
\begin{equation}\label{E:T-def}
T_i =\sum_{j=1}^k t_{i,j},\qquad t_{i,j}\sim \log\lr{\frac{1}{n}\chi_{n-j+1}^2}\quad iid.
\end{equation}
Hence, we find in particular that
\begin{equation}\label{E:T-mean}
    \E{\frac{1}{N}\log\norm{X_{N,n}(\Theta)}}= \sum_{j=1}^k \mu_{n,j}
\end{equation}
and see that Proposition \ref{P:pointwise} is equivalent to showing that for any $s>0$
\begin{equation}\label{E:pointwise-goal-2}
\mathbb P\lr{\abs{\frac{1}{nN}\sum_{i=1}^N \overline{T}_i}\geq s}\leq 2 \exp\lr{-cnN\min\set{M_ks,\frac{s^2}{\xi_{n,k}}}},
\end{equation}
where for any random variable $Y$ we will use the shorthand
\[\overline{Y}:=Y- \E{Y}.\]
We will obtain \eqref{E:pointwise-goal-2} by Markov's inequality applied to certain moments of the sum of the $T_i$'s. Specifically, we will prove the following estimate. 
\begin{proposition}\label{P:moments}
There exists a universal constant $C$ so that for any $n,N,k$ and $p\geq 1$ 
\[\lr{\E{\bigg|\sum_{i=1}^N \overline{T}_i\bigg|^p}}^{1/p} \leq C\lr{\sqrt{pN \sum_{j=1}^k \frac{1}{M_j}} + \frac{p}{M_k}}.\]
\end{proposition}
The proof of Proposition \ref{P:moments}, which is straightforward but tedious, is given in \S \ref{S:moments-pf} below. We assume it for now and complete the proof of \eqref{E:pointwise-goal-2}. Write
\begin{equation}\label{E:p_0-def}
    p_0 := M_k^2 \sum_{j=1}^k \frac{1}{M_j}
\end{equation}
and note that 
\[p\leq Np_0\quad \Longleftrightarrow \quad  \frac{p}{M_k}\leq \sqrt{pN \sum_{j=1}^k \frac{1}{M_j}}.\]
Thus, applying Markov's inequality to Proposition \ref{P:moments} shows that there exists $C>0$ so that for $1\leq p \leq Np_0$
\[\mathbb P\lr{\bigg|\frac{1}{nN}\sum_{i=1}^N \overline{T}_i\bigg| \geq \frac{C}{n\sqrt{N}}\sqrt{p\sum_{j=1}^k\frac{1}{M_j}}}\leq e^{-p}.\]
Equivalently, recalling that
\[\xi_{n,k}=\frac{1}{n}\sum_{j=1}^k \frac{1}{M_j},\]
we see that there exists $c>0$ so that
\[\mathbb P\lr{\bigg|\frac{1}{nN}\sum_{i=1}^N \overline{T}_i\bigg| \geq s}\leq 2e^{-cnN\frac{s^2}{\xi_{n,k}}},\qquad 0\leq s \leq C \xi_{n,k}M_k.\]
This establishes \eqref{E:pointwise-goal-2} in this range of $s.$ To treat $s\geq CM_k\xi_{n,k}$, we again apply Markov's inequality to Proposition \ref{P:moments} to see that there exists $C>0$ so that 
\[p\geq N p_0\quad \Longrightarrow\quad \mathbb P\lr{\bigg|\sum_{i=1}^N \overline{T}_i \bigg| > C\frac{p}{M_k}}\leq e^{-p}.\]
Hence, there exists $c>0$ so that
\[\mathbb P\lr{\bigg|\frac{1}{nN}\sum_{i=1}^N \overline{T}_i\bigg|\geq s}\leq e^{-cnNM_ks},\quad s \geq CM_k \xi_{n,k},\]
completing the proof of \eqref{E:pointwise-goal-2}. %Let us now explain why \eqref{E:pointwise-goal-2} implies \eqref{E:pointwise-goal-1}. We will need the following result. 
%\begin{lemma}\label{L:mean-mu-xi}
%There exists $C>0$ so that for all $n,N,k$ we have
%\[\abs{\E{\frac{2}{nN}\log\norm{X_{N,n}(\Theta)}}-\mu_{n,k}+\xi_{n,k}}\leq \frac{C}{n}\sum_{j=1}^k \lr{\frac{1}{n-j+1}}^2,\]
%where the implied constant is universal. 
%\end{lemma}

\subsection{Proof of Proposition \ref{P:moments}}\label{S:moments-pf}
We seek to estimate the moments of
\[\sum_{i=1}^N \overline{T_i} = \sum_{i=1}^N \sum_{j=1}^k \overline{t}_{i,j},\qquad t_{i,j}\sim \log(\frac{1}{n}\chi_{n-j+1}^2).\]
By Theorem \ref{Latala}, we have
\begin{equation}\label{E:Latala-iid}
\lr{\E{\bigg|\sum_{i=1}^N \overline{T}_i\bigg|^p}}^{1/p}\simeq \sup_{\max\set{2,\tfrac{p}{N}}\leq s \leq p}\frac{p}{s}\lr{\frac{N}{p}}^{\frac{1}{s}}\E{\abs{\overline{T}_i}^s}^{\frac{1}{s}}.
\end{equation}
Our strategy is therefore to estimate $\E{\abs{\overline{T}_i}^s}^{1/s}$ and optimize over $s$. In particular, we seek to show that there exists $C>0$ so that for every $p\geq 1$, we have
\begin{equation}\label{E:moments-goal}
\lr{\E{\abs{\overline{T}_i}^p}}^{1/p}\leq C\lr{\sqrt{p\sum_{j=1}^k M_j^{-1}} +\frac{p}{M_k}}.
\end{equation}
Our first step towards showing \eqref{E:moments-goal} is to obtain the following estimates on the moments of $\overline{t}_{i,j}$.
\begin{lemma}\label{L:tij-est}
There exist $C_1>0$ such that
\[ \lr{\E{\abs{\overline{t}_{i,j}}^p}}^{1/p}\leq C_1 \max\set{\sqrt{\frac{p}{M_j}}, \frac{p}{M_j}}.\]
\end{lemma}
\begin{proof}
We first make a reduction. Namely, let us check that the estimates in Lemma \ref{L:tij-est} for $\overline{t}_{i,j}= t_{i,j}-\E{\log(n^{-1}\chi_{M_j}^2)}$ follow from the same estimates for 
\[
\widehat{t}_{i,j}:=t_{i,j}-\log(n^{-1}M_j).
\]
To see this, recall that by \eqref{E:mu-def-second} and \eqref{E:mu-est}, we have
\[
\E{\log\lr{\frac{1}{n}\chi_{M_j}^2}} = \log\lr{\frac{2}{n}} + \psi\lr{\frac{M_j}{2}} = \log\lr{\frac{M_j}{n}}+ \eps_j,\qquad \eps_j = O(M_j^{-1})
\]
where $\psi$ is the digamma function, and we have used its asymptotic expansion $\psi(z)\sim \log(z)+O(z^{-1})$ for large arguments. Thus, we have for each $i$ that
\begin{align*}
\E{\abs{\overline{t}_{i,j}}^p} &= \E{\abs{\widehat{t}_{i,j}+\eps_j}^p}\leq \sum_{k=0}^p\binom{p}{k} \E{\abs{\widehat{t}_{i,j}}^k} \abs{\eps_j}^{p-k}. 
\end{align*}
So assuming that $\widehat{t}_{i,j}$ satisfy the conclusion of Lemma \ref{L:tij-est}, we find
\begin{align*}
\E{\abs{\overline{t}_{i,j}}^p}&\leq \sum_{k=0}^p\binom{p}{k}\zeta_{k,j}^k\abs{\eps_j}^{p-k}, \qquad \zeta_{k,j}:=C\max\set{\sqrt{\frac{k}{M_j}},\frac{k}{M_j}}.
\end{align*}
Since for $0\leq k \leq p$ we have $\zeta_{k,j}\leq \zeta_{p,j}$, we see that
\[
\E{\abs{\overline{t}_{i,j}}^p}\leq \sum_{k=0}^p\binom{p}{k}\zeta_{p,j}^k\abs{\eps_j}^{p-k} \leq \lr{\zeta_{p,j}+\abs{\eps_j}}^p.
\]
Finally, since $\eps_j=O(M_j^{-1})=O(\zeta_{p,j})$, we find that there exists $C>0$ so that
\[
\lr{\E{\abs{\overline{t}_{i,j}}^p}}^{1/p}\leq  C \zeta_{p,j}\leq C\max\set{\sqrt{\frac{p}{M_j}},\, \frac{p}{M_j}},
\]
as desired. It therefore remains to show that $\widehat{t}_{i,j}=t_{i,j} - \log\lr{n^{-1}M_j}$ satisfies the conclusion of Lemma \ref{L:tij-est}. 
To do this, we begin by checking that there exists $c_1>0$ so that, with $M_j=n-j+1$, for all $s\geq 0$
\begin{equation}\label{E:tij-est}
     \mathbb P\lr{\abs{t_{i,j} - \log\lr{\frac{M_j}{n}}}\geq s}\leq 4e^{-c_1M_j\min\set{s,s^2}}.
\end{equation}
We have
\begin{align*}
    \mathbb P\lr{\abs{t_{i,j} - \log\lr{\frac{M_j}{n}}}\geq s } &= \mathbb P\lr{\abs{\log\lr{\frac{1}{n}\chi_{M_j}^2} - \log\lr{\frac{M_j}{n}}}\geq s }\\
    &=\mathbb P\lr{\abs{\log\lr{\frac{1}{M_j}\chi_{M_j}^2}}\geq s }\\
    &=\mathbb P\lr{\chi_{M_j}^2\geq M_j e^s }+\mathbb P\lr{\chi_{M_j}^2\leq M_j e^{-s} }.
\end{align*}
Let us first bound $\mathbb P\lr{\chi_{M_j}^2\geq M_j e^s }$. Note that the mean of $\chi_{M_j}^2$ is $M_j$ and that $\chi_{M_j}^2 - M_j$ is a sub-exponential random variable with parameters $(4M_j,4)$. Thus, Bernstein's tail estimates for sub-exponential random variables yield the existence of $c>0$ such that for all $t\geq 0$
\begin{equation}\label{E:bernstein}
\mathbb P\lr{\abs{M_j^{-1}\chi_{M_j}^2 - 1}\geq t}\leq 2e^{-cM_j\min\set{t,t^2}}.
\end{equation}
In particular, 
\[
\mathbb P\lr{\chi_{M_j}^2 \geq M_je^s} =\mathbb P\lr{M_j^{-1}\chi_{M_j}^2-1 \geq e^s-1} \leq 2 e^{-cM_j\min\set{e^s-1,(e^s-1)^2}} \leq   2 e^{-cM_j\min\set{s,s^2}},
\]
where in the last inequality we used that $e^s-1 \geq s.$ This gives the first half of \eqref{E:tij-est}.

%{\color{red}
%imply that 
%$$\mathbb P \lr {\abs{ \chi_{M_{j}  }- \sqrt{M_{j}} } \geq t %\sqrt{M_{j}} } \leq 2 e^{ - c \min\{ t, t^{2}\} M_{j}}  , \  t>0 $$
%}
Let us now obtain a similar estimate on $\mathbb P\lr{\chi_{M_j}^2\leq M_j e^{-s} }$, which is a small ball estimate for a sum of iid random variables with bounded density. We need to consider two cases. Theorem \ref{T:iid-small-ball} on small ball estimates for sums of iid random variables shows that there exists a universal constant $C>0$ so that
\[
\mathbb P\lr{\chi_{M_j}^2\leq M_j e^{-s} } \leq (Ce^{-s/2})^{M_j}.
\]
Hence, 
\[
s > s_*:=2\log(C)\quad \Rightarrow \quad \exists c>0 \text{ s.t. } \mathbb P\lr{\chi_{M_j}^2\leq M_j e^{-s} }\leq e^{-csM_j},
\]
giving the desired estimate in this range. Finally, suppose $s\leq s_*$. Then,
\[
e^{-s} \leq 1- e^{-s_*}s
\]
since both sides equal $1$ at $s=0$ and the derivative $-e^{-s}$ of the left hand side is more negative than the derivative $-e^{-s_*}$ of the right hand side for all $s\in (0,s_* )$. Thus, we find for $s\in (0,s_*)$ that
\[
\mathbb P\lr{\chi_{M_j}^2\leq M_j e^{-s} } \leq \mathbb P\lr{\chi_{M_j}^2\leq M_j (1-e^{-s_*}s)} \leq \mathbb P\lr{\abs{M_j^{-1}\chi_{M_j}^2-1}> e^{-s_*}s}
\]
Using the Bernstein inequality \eqref{E:bernstein}, we obtain that there exists $c>0$ depending only on $s_*$ such that for $s\in (0,s_*)$,
\[
\mathbb P\lr{\chi_{M_j}^2\leq M_j e^{-s} } \leq 2e^{-cM_j\min\set{s,s^2}},
\]
giving the desired estimate in this range as well. This establishes \eqref{E:tij-est}. 
To complete the proof that $\widehat{t}_{i,j}$ satisfy Lemma \ref{L:tij-est}, we use \eqref{E:tij-est} to write
\begin{align*}
\E{\abs{\widehat{t}_{i,j}}^p}&=\int_0^\infty \mathbb P\lr{\abs{\widehat{t}_{i,j}}>x}px^{p-1}dx\\
&\leq p \left[\int_0^1 e^{-cM_jx^2} x^{p-1}dx + \int_1^\infty e^{-cM_jx}x^{p-1}dx\right].
\end{align*}
The first term can be estimated by comparing to the moments of a Gaussian as follows:
\begin{align*}
p \int_0^1 e^{-cM_jx^2} x^{p-1}dx &= p\lr{ 2cM_j}^{-p/2} \int_0^{(2cM_j)^{1/2}} e^{-x^2}x^{p-1}dx\\
&\leq p\lr{ 2cM_j}^{-p/2} \int_0^{\infty} e^{-x^2}x^{p-1}dx\\
&\leq p\lr{ 2cM_j}^{-p/2} 2^{\frac{p}{2}}\Gamma\lr{\frac{p}{2}}\\
&\leq p\lr{ \frac{p}{2cM_j}}^{p/2},
\end{align*}
where we used that for $z>0$ we have $\Gamma(z)\leq z^z.$ The second term can similarly be estimated by comparing the moments of an exponential:
\begin{align*}
p \int_1^\infty e^{-cM_jx} x^{p-1}dx &= p (cM_j)^{-p}\int_{cM_j}^\infty e^{-x}x^{p-1}dx\\
&=p(cM_j)^{-p} (p-1)!\\
&\leq p\lr{\frac{p}{cM_j}}^p.
\end{align*}
Putting these two estimates together and taking $1/p$ powers, we find that there exists $C>0$ so that
\[ \lr{\E{\abs{\widehat{t}_{i,j}}^p}}^{1/p}\leq C \max\set{\sqrt{\frac{p}{M_j}}, \frac{p}{M_j}}\]
for all $p\geq 1.$ This completes the proof. 
\end{proof}

\noindent With Lemma \ref{L:tij-est} in hand, we are now in a position to show \eqref{E:moments-goal}. Since
\[\overline{T}_i=\sum_{j=1}^k \overline{t}_{i,j},\]
we have by Theorem \ref{Latala} that
\begin{equation}\label{E:Latala-noniid}
\left( \mathbb E |\overline{T}_i|^{p} \right)^{\frac{1}{p}} \simeq  \inf\left\{ t>0 : \sum_{j=1}^{k} \log \left[ \mathbb E | 1+ \frac{\overline{t}_{i,j}}{ t}|^{p} \right] \leq p \right\},    
\end{equation}
where $\simeq$ means bounded above and below by absolute constants. We will use the notation from \eqref{E:p_0-def}:
\[p_0=M_k^2\sum_{j=1}^k M_j^{-1}.\]
Since
\[\sqrt{p\sum_{j=1}^k M_j^{-1}} \leq  \frac{p}{M_k}\qquad \Longleftrightarrow\qquad p\geq p_0,\]
we will show \eqref{E:moments-goal} by breaking into two cases. Namely, we will show that there exists $C>0$ so that
\begin{equation}\label{E:moments-goal-1}
p\leq p_0\quad \Longrightarrow\quad \left( \mathbb E \left| \overline{T}_i \right|^{p} \right)^{\frac{1}{p}} \leq C \sqrt{p \sum_{j=1}^{k} M_{j}^{-1} }  
\end{equation}
as well as
\begin{equation}
\label{E:moments-goal-2}
p\geq p_0\quad\Longrightarrow\quad\left( \mathbb E \left| \overline{T}_i \right|^{p} \right)^{\frac{1}{p}} \leq C\frac{p}{M_{k}}. 
\end{equation}
We may assume without loss of generality that $p$ is an even integer. Indeed, once we've show \eqref{E:moments-goal-1} and \eqref{E:moments-goal-2} for even integers $p$ (and a uniform constant $C$), we may use that 
\[\lr{\E{\abs{\overline{T}_i}^p}}^{1/p}\leq \lr{\E{\abs{\overline{T}_i}^{p'}}}^{1/p'},\]
where $p'$ is the smallest even integer greater than or equal to $p.$ This yields \eqref{E:moments-goal-1} and \eqref{E:moments-goal-2} for any  $p$ with $C$ replaced by $2C.$ 

We now turn to showing that \eqref{E:moments-goal-1} holds with $p$ an even integer. To do this, we will need to evaluate the expectation $\E{|1+t^{-1}t_{i,j}|^p}$ appearing in \eqref{E:Latala-noniid}. A key point is to use that $\overline{t}_{i,j}$ are centered. Since $p$ is even, we may bring this to bear most directly by noting that the absolute value in $\E{|1+t^{-1}t_{i,j}|^p}$ is unnecessary and using that $\E{\overline{t}_{i,j}}=0$. Lemma \ref{L:tij-est} and the well-known estimate
\begin{equation}\label{E:binom-est}
\lr{\frac{n}{k}}^k\leq \binom{n}{k}\leq \lr{\frac{n}{k}}^k e^k,\qquad k\geq 1
\end{equation}
yield that for all $i,j$ we have:
\begin{align}
\notag \E{\lr{ 1 + \frac{ \overline{t}_{i,j}}{ t} }^{p}} &= 1 + \sum_{\ell=2}^{p} { p\choose \ell} \frac{ \E{ \overline{t}_{i,j}^{\ell}} }{ t^{\ell}}\\
\notag &\leq 1 + \sum_{\ell=2}^{\min\set{p,M_j}} \binom{p}{\ell} \lr{\frac{C\ell}{t^2M_j}}^{\ell/2}+ \sum_{\ell=M_j+1}^p\binom{p}{\ell} \lr{\frac{C\ell}{t M_j}}^\ell\\
\notag &\leq 1 + \sum_{\ell=2}^{\min\set{p,M_j}} \lr{\frac{p}{\ell}}^\ell \lr{\frac{C\ell}{t^2M_j}}^{\ell/2}+ \sum_{\ell=M_j+1}^p \lr{\frac{Cp}{t M_j}}^\ell\\
\label{E:t-t-est}&\leq 1 + \sum_{\ell=2}^{\min\set{p,M_j}} \lr{\frac{p}{\ell}}^{\ell/2} \lr{\frac{Cp}{t^2M_j}}^{\ell/2}+ \sum_{\ell=M_j+1}^p \lr{\frac{Cp}{t M_j}}^\ell.
\end{align}
We now bound the first two terms in the previous line by breaking into the terms where $\ell$ is even and odd. When $\ell$ is even the terms in \eqref{E:t-t-est} are bounded above by
\begin{align}
\label{E:even-est-1}   1+ \sum_{\substack{\ell=2\\ \ell\text{ even}}}^{\min\set{p,M_j}}  \lr{\frac{p}{\ell}}^{\ell/2} \lr{\frac{C'p}{t^2M_j}}^{\ell/2}&\leq 1+\sum_{\substack{\ell=2\\ \ell\text{ even}}}^{\min\set{p,M_j}} \binom{p/2}{\ell/2} \lr{\frac{C'p}{t^2 M_j}}^{\ell/2}\leq  \lr{1+\frac{C'p}{t^2M_j}}^{p/2}.
\end{align}
To bound the odd terms in \eqref{E:t-t-est}, let us first note that for any $0\leq m \leq \ell\leq p$, we have
\[
\lr{\frac{p}{\ell}}^\ell \leq p^m \binom{p}{\ell-m}.
\]
Indeed, when $\ell=m$, this equality simply reads
\[
\lr{\frac{p}{\ell}}^{\ell} \leq p^{\ell+1},
\]
while if $m<\ell,$ we note that the inequality is equivalent to 
\[
\frac{p^{\ell-m}}{\ell^\ell} \leq \binom{p}{\ell-m},
\]
which follows by estimating the expression on the left hand side using \eqref{E:binom-est} as follows:
\[
\lr{\frac{p}{\ell}}^{\ell-m} \frac{1}{\ell^{m}}=\lr{\frac{p}{\ell-m}}^{\ell-m}\lr{\frac{\ell-m}{\ell}}^{\ell-m} \frac{1}{\ell^{m}}\leq \binom{p}{\ell-m}.
\]
Thus, the odd terms in \eqref{E:t-t-est} are bounded above by
\begin{align}
\notag \sum_{\substack{\ell=3\\ \ell\text{ odd}}}^{\min\set{p,M_j}} \lr{\frac{p}{\ell}}^{\ell/2}  \lr{\frac{Cp}{t^2 M_j}}^{\ell/2} &\leq \min_{m=1,3}\set{\lr{\frac{Cp^2}{t^2M_j}}^{\frac{m}{2}}\sum_{\substack{\ell=3\\ \ell\text{ odd}}}^{p-1}   {p\choose \ell-m}^{\frac{1}{2}} \lr{\frac{Cp}{t^2M_j}}^{\frac{\ell-m}{2}}}.
\end{align}
To proceed, note that for any $0\leq b\leq a$
\[
\binom{2a}{2b} \leq 2^b \binom{a}{b}^2.
\]
This inequality follows by observing that for any $j=0,\ldots, b-1$, we have
\[
\frac{(2a-2j)(2a-2j-1)}{(2b-2j-1)} = \frac{(a-j)(a-j-1/2)}{(b-j)(b-j-1/2)}\leq 2 \lr{\frac{a-j}{b-j}}^2 
\]
and repeatedly applying this estimate to the terms in $\binom{2a}{2b}$. Thus, we obtain 
\begin{align}
\notag \sum_{\substack{\ell=3\\ \ell\text{ odd}}}^{\min\set{p,M_j}} \lr{\frac{p}{\ell}}^{\ell/2}  \lr{\frac{Cp}{t^2 M_j}}^{\ell/2} &\leq \min_{m=1,3}\set{\lr{\frac{Cp^2}{t^2M_j}}^{\frac{m}{2}}}\sum_{\ell=0}^{p/2}  \binom{p/2}{\ell}\lr{\frac{Cp}{t^2M_j}}^{\ell}\\
\notag &=\min_{m=1,3}\set{\lr{\frac{Cp^2}{t^2M_j}}^{\frac{m}{2}}} \lr{1+\frac{Cp}{t^2 M_j}}^{p/2}\\
 &\leq \lr{\frac{Cp^2}{t^2M_j}} \lr{1+\frac{Cp}{t^2 M_j}}^{p/2}\label{E:odd-est-1},
\end{align}
where in the last inequality we've used that $\min\set{x^{1/2}, x^{3/2}}\leq x$ for all $x\geq 0.$ Putting together \eqref{E:even-est-1} and
\eqref{E:odd-est-1} we see that there exists $C>0$ so that
\begin{equation}\label{E:gen-moment-est}
    \E{ \lr{1+\frac{\overline{t}_{i,j}}{t}}^p}\leq \lr{1+\frac{Cp^2}{t^2 M_j}}\lr{1+\frac{Cp}{t^2M_j}}^{p/2}+ \sum_{\ell=M_j+1}^p \lr{\frac{Cp}{t M_j}}^\ell.
\end{equation}
Let us now verify \eqref{E:moments-goal-1}. Recall that for $j\leq k$ we have $M_{j} \geq M_{k}$ and that  $p\leq p_{0} = M_{k}^{2} \sum_{j=1}^{k} \frac{1}{M_{j}}$.  We set
\begin{equation}\label{E:t-setting}
    t = \sqrt{C'p \sum_{i=1}^k M_i^{-1}}
\end{equation}
where 
\[
C^{\prime}= \max\set{\lr{16C}^2,\, 2C^{1/2}}
\]
is an absolute constant depending only on the constant $C$ appearing in \eqref{E:gen-moment-est}. For this choice of $C'$, we have
\[\frac{Cp}{tM_j} = C\sqrt{\frac{p^2}{C'pM_j^{2} \sum_{i=1}^kM_i^{-1}}} \leq   C\sqrt{\frac{p}{C'M_k^{2} \sum_{i=1}^kM_i^{-1}}} = C\sqrt{\frac{p}{C'p_0}}\leq \frac{1}{16}, \forall j \leq k \]
and 
$$ \sum_{j=1}^{k} \frac{Cp}{ t^{2} M_{j}} \leq \frac{C}{(C^{\prime})^{2}} \leq \frac{1}{4}.$$
In particular for $j\leq k $ such that $ M_{j} \leq p $, 
$$ \sum_{\ell= M_{j}+ 1}^{p} \left( \frac{Cp}{ t M_{j}}\right)^{\ell} \leq \sum_{\ell= M_{j}+ 1}^{p} \frac{1}{ 16^{\ell}}  \leq  \frac{1}{ 4^{M_{j}}} $$ 
Hence, since $ \log{(a+b)} \leq (\log{a}) + b $ for $ a\geq 1$ and $ b>0$, 
\begin{align}
\notag\sum_{j=1}^k\log \E{\lr{1+\lr{\frac{\overline{t}_{i,j}}{t}}^p}} &\leq \frac{p}{2}\sum_{j=1}^k \log\lr{1+\frac{Cp}{t^2M_j}} + \sum_{j=1}^k\log\lr{1+\frac{Cp^2}{t^2M_j}}+\sum_{j=1}^k \lr{\frac{1}{4}}^{M_j}\\
\notag&\leq \frac{p}{2}\sum_{j=1}^k \frac{Cp}{t^2M_j} + \sum_{j=1}^k \frac{Cp^2}{t^2M_j}+ \sum_{s= n-k+1}^{n} \frac{1}{4^{s}} \\
\notag&= \frac{3p}{8}+\frac{1}{2}\\
\notag&\leq p.
\end{align}
Hence, \eqref{E:moments-goal-1} follows from \eqref{E:Latala-noniid}. We now turn to the case when $p\geq p_0$ and seek to show \eqref{E:moments-goal-2}. Rather than \eqref{E:t-setting}, to show \eqref{E:moments-goal-2}, we set
\begin{equation}\label{E:t-setting-2}
    t = \frac{C'p}{M_k}
\end{equation}
with 
\[
C' = \max\set{4C,\, 2C^{1/2}}.
\]
Then,
\[\frac{Cp}{tM_j} = \frac{CM_k}{C'M_j}<\frac{C}{C'}\leq \frac{1}{4}\]
 and
\[\sum_{j=1}^k \frac{Cp}{t^2M_j} = \sum_{j=1}^k \frac{CM_k^2}{(C')^2pM_j}= \frac{C}{(C')^2} \frac{p_0}{p}\leq \frac{C}{(C')^2}\leq\frac{1}{4} .\]
Hence from \eqref{E:gen-moment-est}, we find
\begin{align*}
\sum_{j=1}^k\log \E{\lr{1+\lr{\frac{\overline{t}_{i,j}}{t}}^p}} &\leq \frac{p}{2}\sum_{j=1}^k \log\lr{1+\frac{Cp}{t^2M_j}} + \sum_{j=1}^k\log\lr{1+\frac{Cp^2}{t^2M_j}}+\sum_{j=1}^k \lr{\frac{1}{4}}^{M_j}\\
&\leq \frac{p}{2}\sum_{j=1}^k \frac{Cp}{t^2M_j} + \sum_{j=1}^k\frac{Cp^2}{t^2M_j}+\frac{1}{2}\\
&\leq \frac{3Cp}{2(C')^2} +\frac{1}{2}\\
&\leq p.
\end{align*}
Thus, we see that relation \eqref{E:moments-goal-2} also follows from \eqref{E:Latala-noniid}. We are now in a position to finish the proof of Proposition \ref{P:pointwise} by combining \eqref{E:Latala-iid} with \eqref{E:moments-goal-1} and \eqref{E:moments-goal-2}. We find from \eqref{E:Latala-iid} that
\begin{align}
\notag    \lr{\E{\abs{\sum_{i=1}^N \overline{T}_i}^p}}^{1/p} &\leq \sup_{\max\set{2,\frac{p}{N}}\leq s \leq p} \frac{p}{s}\lr{\frac{N}{p}}^{1/s}\E{\abs{\overline{T}_i}^s}^{1/s}.
\end{align}
If $p\leq p_0$, then \eqref{E:moments-goal-1} implies
\begin{align}
\label{E:p-small}   \lr{\E{\abs{\sum_{i=1}^N \overline{T}_i}^p}}^{1/p} &\leq C \sup_{\max\set{2,\frac{p}{N}}\leq s \leq p} \frac{p}{s}\lr{\frac{N}{p}}^{1/s}\sqrt{s\sum_{j=1}^kM_j^{-1}}.
\end{align}
Define
\[f(s)=\log\lr{\frac{p}{s}\lr{\frac{N}{p}}^{1/s}\sqrt{s}}=\log(p) - \frac{1}{2}\log(s) + \frac{1}{s}\log\lr{\frac{N}{p}}.\]
Note that
\[f'(s) = -\frac{1}{2s}-\frac{1}{s^2}\log\lr{\frac{N}{p}}.\]
When $p\leq N$ the function $f$ is manifestly monotone decreasing in $s>0$. Therefore, taking $s=2$, we find
\[
p\leq N \quad \Rightarrow \quad  \lr{\E{\abs{\sum_{i=1}^N \overline{T}_i}^p}}^{1/p} \leq C \sqrt{pN\sum_{j=1}^k M_j^{-1}}
\]
On the other hand, when $p\geq N$, we have
\[
\frac{p}{s}\lr{\frac{N}{p}}^{1/s}\sqrt{s\sum_{j=1}^k M_j^{-1}} \leq ps^{-1/2}\sqrt{\sum_{j=1}^k M_j^{-1}},
\]
which is strictly decreasing in $s>0$. Hence, taking $s=p/N$, we again find
\[
p\geq N \quad \Rightarrow \quad  \lr{\E{\abs{\sum_{i=1}^N \overline{T}_i}^p}}^{1/p} \leq C \sqrt{pN\sum_{j=1}^k M_j^{-1}}.
\]
Hence, we find that if $p\leq p_0$, then
\begin{align*}   \lr{\E{\abs{\sum_{i=1}^N \overline{T}_i}^p}}^{1/p} &\leq C  \sqrt{pN\sum_{j=1}^kM_j^{-1}},
\end{align*}
as desired. Finally, if $p\geq p_0$, then \eqref{E:moments-goal-2} implies
\begin{align}
\label{E:p-small}   \lr{\E{\abs{\sum_{i=1}^N \overline{T}_i}^p}}^{1/p} &\leq C \sup_{\max\set{2,\frac{p}{N}}\leq s \leq p} \frac{p}{s}\lr{\frac{N}{p}}^{1/s}\frac{s}{M_k}=\frac{Cp}{M_k} \sup_{\max\set{2,\frac{p}{N}}\leq s \leq p} \lr{\frac{N}{p}}^{1/s}.
\end{align}
Note that if $p/N\geq 1$, then for every $s\geq 2$ we have $(N/p)^{1/s}\leq 1$ and hence 
\[\frac{Cp}{M_k}\sup_{\max\set{2,\frac{p}{N}}\leq s \leq p} \lr{\frac{N}{p}}^{1/s}\leq \frac{Cp}{M_k}.\]
Further, if $p/N\leq 1$, then $(N/p)^{1/s}$ is monotonically decreasing with $s$ and hence
\[\frac{Cp}{M_k}\sup_{\max\set{2,\frac{p}{N}}\leq s \leq p} \lr{\frac{N}{p}}^{1/s}\leq \frac{Cp}{M_k}\lr{\frac{N}{p}}^{1/2}=C\sqrt{pNM_k^{-2}}\leq C\sqrt{pN\sum_{j=1}^k M_j^{-1}}.\]
Thus, in all cases we find
\[\lr{\E{\bigg|\sum_{i=1}^N \overline{T}_i\bigg|^p}}^{1/p} \leq C\lr{\sqrt{pN \sum_{j=1}^k \frac{1}{M_j}} + \frac{p}{M_k}},\]
which is precisely the statement of Proposition \ref{P:moments}.\hfill $\square$

%%%%%%%%%%%%%%%%%%%%%%%%%%%%%%%%%%%%%%%%%%%%%%
\section{Lyapunov Sums: Proof of Theorem \ref{T:sums}}\label{S:sums-pf}
%%%%%%%%%%%%%%%%%%%%%%%%%%%%%%%%%%%%%%%%%%%%%%

To prove Theorem \ref{T:sums}, we start by combining previous estimates for $\frac{1}{N}\log\norm{X_{N,n}(\Theta)}$ from Proposition \ref{P:pointwise} with the deviation estimates in Proposition \ref{P:pointwise-small-ball}. Let recall the definition of the function $g$ (cf \eqref{E:g-def})
\begin{equation}
%\label{def-g}
g_{n,k}(s)  = \left\{ \begin{array}{cc} 
    \min\left\{ 1, \frac{ n s}{ k}\right\} \  ,& \hspace{5mm} k\leq \frac{n}{2} \\
    \min\left\{ \delta_{n,k}, \frac{ s}{ \log{1/\delta_{n,k}}} \right\} \  ,& \hspace{5mm} k>\frac{n}{2} \\
\end{array} \right.,
\end{equation}
where we recall that for $ k \geq \frac{n}{2}$ 
\[ \delta_{n,k}:= \frac{ n-k+1}{n}.\]
Let $ s>0$ and  $ 1\leq k<m  \leq n. $
Writing 
$$  p_{s,k, m} := \mathbb P \left( \abs{ \frac{1}{n} \sum_{i=m}^{k}\lr{ \lambda_{i} -\mu_{n,i}} } \geq s \right), $$
we seek to show that there exist universal constants $c_1,c_2,c_3>0$ such that for any $ 1\leq m \leq k \leq n $ and every $ s\geq c_1 \frac{k}{nN} \log{\frac{en}{k}} $ we have
\begin{equation}
\label{conc-L-1-new}
p_{s,k,m} \leq c_2 \exp{ \left\{ - c_3\min\set{nNs,  n^2 N s   g_{n,k}(s) } \right\}}.
\end{equation}
To see this, note that the triangle inequality yields
\[p_{s,k,m}\leq p_{s/2,k,1} + p_{s/2,m-1,1}.\]
%Observe also that for each $s,$ the function $g_{n,k}(s)$ is monotonically decreasing in $k$ on each of $[0,n/2]$ and on $[n/2,n]$ and that moreover there exists $c>0$ so that $g_{n,n/2-1}(s)\geq cg_{n,n/2+1}.$
Hence, it suffices to prove \eqref{conc-L-1-new} with $m=1.$ To do this, fix $\Theta \in \mathrm{Fr}_{n,k}$, an orthonormal $k-$frame in $\R^n$. We may write for any $s>0$
\begin{align}
\label{E:prob-1}    p_{s,k, 1} & \leq \mathbb P\lr{\abs{\frac{1}{n}\sum_{i=1}^k \lambda_i - \frac{1}{nN}\log\norm{X_{N,n}(\Theta)}}\geq \frac{s}{2}}\\
\label{E:prob-2}    &+ \mathbb P\lr{\abs{\frac{1}{nN}\log\norm{X_{N,n}(\Theta)} - \frac{1}{n}\sum_{i=1}^k \mu_{n,i}}\geq \frac{s}{2}}.
\end{align}
We will show separately that the probabilities in \eqref{E:prob-1} and \eqref{E:prob-2} are both bounded above by the right hand side of \eqref{conc-L-1-new}. To check this for \eqref{E:prob-2}, note that Proposition \ref{P:pointwise} guarantees 
\begin{equation}
\label{proposition9-1-new-00}
\mathbb P\left( \left| \frac{  \log{ \| X_{N,n} ( \Theta)\|  }  }{ nN } -  \frac{1}{n} \sum_{j=1}^{k} \mu_{n,j}  \right|\geq s \right) \leq 2 \exp\left\{ - c n N \min\{ M_{k} s , \frac{s^{2}}{ \xi_{n,k}}\} \right\}   , s>0 , 
\end{equation}
where we remind the reader that
$$ M_{j} := n-j+1 , \  \xi_{n,k} := \frac{1}{n} \sum_{j=1}^{k} \frac{ 1}{M_{j}} , \ \mu_{n,k} := \frac{1}{2}\mathbb E \left[ \log\left( \frac{1}{n} \chi_{n-j+1}^{2}\right)\right] .$$
Some routine algebra reveals 
\begin{equation}
\label{xi-mu-estimates-1}
k\leq \frac{n}{2}\qquad \Rightarrow \qquad n\xi_{n,k} \simeq \frac{ k}{n} ,\qquad  M_{k}  \simeq n
\end{equation}
and
\begin{equation}
\label{xi-mu-estimates-1}
k\geq \frac{n}{2}\qquad \Rightarrow \qquad n\xi_{n,k} \simeq \log{\left( \frac{1}{\delta_{n,k}}\right)}  ,\qquad M_{k}=\delta_{n,k}n,
\end{equation}
where $a \simeq b$ means that there exists $c_1,c_2>0$ so that $c_1a\leq b\leq c_2 a.$ Hence, 
\[
k\leq \frac{n}{2}\quad \Rightarrow \quad \min\set{M_ks,\frac{s^2}{\xi_{n,k}}} \simeq ns\min\set{1, \frac{ns}{k}}
\]
and similarly 
\[
k\geq \frac{n}{2}\quad \Rightarrow \quad \min\set{M_ks,\frac{s^2}{\xi_{n,k}}} \simeq ns\set{\delta_{n,k}, \frac{s}{\log(1/\delta_{n,k})}}.
\]
Putting these two estimates together, we find that 
\eqref{proposition9-1-new-00} yields for any $s>0$
\begin{align}
\label{proposition9-1-new-1}
\mathbb P\left(  \left|  \frac{ \log \norm{X_{N,n}(\Theta)} }{ nN } - \frac{1}{n}  \sum_{j=1}^{k} \mu_{n,j}   \right| \geq s \right) \leq 2 \exp\left\{ - c n^{2} N sg_{n,k}(s) \right\},
\end{align}
as desired. Turning to the probability in \eqref{E:prob-1}, recall that in Proposition \ref{P:pointwise-small-ball}, we have shown that for every $ \varepsilon\in (0,1) $, 
 \[\mathbb P\lr{\abs{\frac{1}{nN}\log \norm{X_{N,n}(\Theta)} -  \frac{1}{n}\sum_{i=1}^{k} \lambda_{i} }\geq \frac{k}{2Nn}\log\lr{\frac{n}{k\eps^2}}}\leq \lr{C\eps}^\frac{k}{2}.\]
If we set $ s:= \frac{k}{nN} \log{ \frac{en}{k\varepsilon^{2}}} $, then
\[
(C\eps)^{k/2} = \exp\left[-\frac{1}{4}snN + \frac{k}{4}\log\lr{\frac{en}{k}}+\frac{k}{2}\log(C)\right].
\]
Hence, assuming that
\[
s\geq C'\frac{k}{nN}\log\lr{\frac{en}{k}}
\]
for $C'$ sufficiently large, we arrive to the following expression:
\begin{equation}
\label{conc-by-sb}
\mathbb P \left(\abs{  \frac{1}{n} \sum_{i=1}^{k} \lambda_{i} - \frac{ \log \norm{X_{N,n}(\Theta)}}{nN} }\geq  s  \right) \leq e^{ - \frac{snN }{4}}, \quad s\geq C'\frac{ k}{nN} \log\lr{ \frac{en}{k}}.
\end{equation}
Thus, putting together \eqref{proposition9-1-new-1} and \eqref{conc-by-sb}, we find that
\[
p_{s,k,1}\leq c_2 \exp\set{-c_3 \min\set{nNs, n^2 Ns g_{n,k}(s)}},
\]
completing the proof. \hfill $\square$

\section{Convergence to the Triangle Law: Proof of Theorem \ref{T:global}}\label{S:global-pf}
In this section, we derive Theorem \ref{T:global} from Theorem \ref{T:sums}. We will need the following elementary result.

\begin{lemma}\label{L:digamma}
Fix positive integers $n,q,m$ satisfying $4\leq m\leq q\leq n$.
Then
\[\frac{m}{2n}\log\lr{q/n} - \frac{1}{n}\sum_{j=n-q+1}^{n-q+m} \mu_{n,j}\geq \frac{(m-1)^2}{4nq}.\]
Further, assuming that $n-q-m\geq 0$, we also have
\[\frac{m}{2n}\log(q/n)-\frac{1}{n}\sum_{j=n-q-m+1}^{n-q}\mu_{n,j}\leq -\frac{m(m-3)}{3nq}.\]
\end{lemma}
\begin{proof}
Let us first prove the lower bound. Recall that
\begin{equation}\label{E:mu-form}
    \mu_{n,j}=\frac{1}{2}\lr{\log\lr{\frac{2}{n}}+\psi\lr{\frac{n-j+1}{2}}}.
\end{equation}
Moreover, a well-known estimate \cite[eq.6.3.18, p.259]{abramowitz+stegun} for the digamma function is:
\[\psi(x)< \log(x).\]
Using this we obtain
\[\mu_{n,j}< \frac{1}{2}\log\lr{1-\frac{j-1}{n}} ,\]
which allows us to write
\begin{align*}
    \frac{m}{2n}\log\lr{q/n}- \frac{1}{n}\sum_{j=n-q+1}^{n-q+m} \mu_{n,j} &\geq \frac{1}{2n}\sum_{j=n-q+1}^{n-q+m}\log\lr{\frac{q}{n-j+1}}=\frac{1}{2n}\sum_{j=1}^{m-1}\log\lr{\frac{1}{1-j/q}}.
\end{align*}
Since $q,n$ are fixed, let us temporarily introduce
\[\xi:=\frac{q}{n}.\]
With this notation, because $\log(1/(1-t))$ is monotonically increasing for $t\in [0,1)$, we have
\[\frac{1}{2n}\sum_{j=1}^{m-1}\log\lr{\frac{1}{1-j/q}}=\frac{1}{2}\sum_{j=1}^{m-1}\frac{1}{n}\log\lr{\frac{1}{1-(j/n)/\xi}}\geq \frac{1}{2}\int_0^{\frac{m-1}{n}} \log\lr{\frac{1}{1-t/\xi}}dt.\]
Some routine calculus therefore reveals that
\begin{align*}
     \frac{m}{2n}\log\lr{q/n}- \frac{1}{n}\sum_{j=n-q+1}^{n-q+m} \mu_{n,j}&\geq \frac{\xi}{2}\left[\lr{1-\eps/\xi}\log\lr{1-\eps/\xi} + \eps/\xi\right],
\end{align*}
where we've set $\eps := (m-1)/n.$ Finally, note that for $x>0$
\[(1-x)\log(1-x) + x = \sum_{k\geq 2} \frac{x^k}{k(k-1)} \geq x^2/2.\]
Hence, we obtain 
\begin{align*}
     \frac{m}{2n}\log\lr{q/n}- \frac{1}{n}\sum_{j=n-q+1}^{n-q+m} \mu_{n,j}&\geq \frac{\eps^2}{4\xi}=\frac{(m-1)^2}{4nq},
\end{align*}
as claimed. Let us now derive the upper bound. Using \eqref{E:mu-form}, we get
\begin{align*}
    \frac{m}{2n}\log\lr{q/n}- \frac{1}{n}\sum_{j=n-q-m+1}^{n-q} \mu_{n,j} &= \frac{1}{2n}\sum_{j=n-q-m+1}^{n-q}\left\{\log\lr{\frac{q}{2}} - \psi\lr{\frac{n-j+1}{2}}\right\}\\
    &= \frac{1}{2n}\sum_{j=1}^{m}\left\{\log\lr{\frac{q}{2}} - \psi\lr{\frac{q+j}{2}}\right\}.
\end{align*}
Using the inequality (see again \cite[eq.6.3.18, p.259]{abramowitz+stegun})
\[\psi(x)>\log\lr{x} - 1/x,\]
we arrive at the estimate
\begin{align*}
    \frac{m}{2n}\log\lr{q/n}- \frac{1}{n}\sum_{j=n-q-m+1}^{n-q} \mu_{n,j} 
    &\leq \frac{1}{2n}\sum_{j=1}^{m}\left\{\log\lr{\frac{q}{q+j}}+\frac{2}{q+j}\right\}.
\end{align*}
As before, we will estimate this sum above by an integral. Still writing $\xi=q/n$, we have as before
\[\frac{1}{2n}\sum_{j=1}^m \log\lr{\frac{1}{1+j/q}} \leq \frac{\xi}{2}\int_0^{\eps/\xi} \log\lr{\frac{1}{1+t}}dt = \frac{\xi}{2}\left[-(1+\eps/\xi)\log\lr{1+\eps/\xi}+  \eps/\xi\right]\]
where we've now set $\eps = m/n$ (which is slightly different than above). For $x\in (0,1)$, we have 
\[-(1+x)\log(1+x)+x = \sum_{k\geq 2}(-1)^{k+1}\frac{x^k}{k(k-1)}\leq -\frac{x^2}{2}+\frac{x^3}{6}\leq -\frac{x^2}{3},\]
we therefore find
\[\frac{1}{2n}\sum_{j=1}^m \log\lr{\frac{1}{1+j/q}}\leq  -\frac{\eps^2}{3\xi} .\]
Next, 
\[\frac{1}{q}\sum_{j=1}^m \frac{1}{n}\frac{1}{1+\xi^{-1}(j/n)}\leq \frac{1}{q}\int_0^{\eps} \frac{dt}{1+t/\xi} =\frac{\xi}{q}\log\lr{1+ \eps/\xi}\leq  \frac{\eps}{q}.\]
So all together we find the upper bound
\[  \frac{m}{2n}\log\lr{q/n}- \frac{1}{n}\sum_{j=n-q-m+1}^{n-q} \mu_{n,j}\leq -\frac{m^2}{3nq}+\frac{m}{nq}=-\frac{m(m-3)}{3nq}.\]
This completes the proof. 
\end{proof}

\noindent We now conclude the proof of Theorem \ref{T:global}. To do this, fix $\eps>0$ and assume that 
\[n>\frac{c_1\sqrt{\log(1/\eps)}}{\eps},\qquad N>\frac{c_2}{\eps^2}\]
for some constants $c_1,c_2>1$ that we will fix later. To prove Theorem \ref{T:global} note that the bound above on $n$  guarantees that $\eps>c_1/n.$ Hence, we need only consider such $\eps.$ Moreover, we may always assume that
\[\eps = \frac{m}{n}\]
for some integer $5\leq m\leq n$ since $\mathcal U(t)$ is $1$-lipschitz, and will use $\eps$ and $m/n$ interchangeably. Next, recall the following notation for the cumulative distributions
\begin{align*}
    \mathcal H_{N,n}(s):=\frac{1}{n}\#\set{j\leq n~\big |~ s_j(X_{N,n})^{2/N} \leq s},\qquad \mathcal U(s):=\begin{cases}0,&\quad s<0\\
    s,&\quad 0\leq s\leq 1\\
    1,&\quad s>1\end{cases}
\end{align*}
of the squared singular values of $X_{N,n}$ and the uniform distribution on $[0,1]$. Let us define the event
\[S_{n,m}:=\set{\forall q\in \set{m+1,\ldots, n}~ \abs{\mathcal H_{N,n}\lr{q/n}-\mathcal U\lr{q/n}}\leq \eps}.\]
On this event, since $\mathcal H_{N,n}$ and $\mathcal U$ are both monotone we have for $t\leq (m+1)/n$ that
\begin{align*}
    \mathcal H_{N,n}(t) &\leq \mathcal H_{N,n}\lr{(m+1)/n}\leq \eps + \mathcal U\lr{(m+1)/n}=\eps + (m+1)/n \leq 3\eps.
\end{align*}
Similarly if $t>1$
\begin{align*}
1-\mathcal H_{N,n}(t)\leq 1-\mathcal H_{N,n}\lr{1} \leq  1-\mathcal U\lr{1}+\eps = \eps
\end{align*}
Using the same idea we may write for any $t\in [(m+1)/n,1]$
\begin{align*}
    \mathcal H_{N,n}(t)- \mathcal U(t) &\leq \mathcal H_{N,n}((j+1)/n) - \mathcal U((j+1)/n) + \mathcal U(t) - \mathcal U((j+1)/n)\\
    &\leq \eps  + 1/n\\
    &\leq 2\eps,
\end{align*}
where $m+1\leq j\leq n$ is the unique integer for which
\[\frac{j}{n}\leq t < \frac{j+1}{n}.\]
Hence, 
\[\mathbb P\lr{\sup_{t\in \R}\abs{\mathcal H_{N,n}(t)- \mathcal U(t)}> 3\eps}\leq \mathbb P\lr{S_{n,m}^c}\]
and we must therefore bound $\mathbb P\lr{S_{n,m}^c}$ from above. We will do this by rewriting all events involving the singular values $s_i$ in terms of the Lyapunov exponents $\lambda_j.$ Moving forward, let us agree that any event that involves $\lambda_{-s}$ or $\lambda_{n+s}$ for $ s>0$ is by definition empty. Since 
\[
s_j(X_{N,n})^{2/N}\leq \frac{q}{n}\quad \Leftrightarrow\quad \lambda_j\leq \frac{1}{2}\log\lr{\frac{q}{n}},
\]
we find
\begin{align*}
    \left| {\cal{H}}_{N,n} \left(q/n\right)  - {\mathcal U}\left( q/n \right)   \right| &= \left|  \frac{1}{n} \#\left\{ j\leq n : \lambda_{j} \leq \frac{1}{2} \log\left(q/n\right) \right\} -   q/n   \right|\\
    &=\left|  \frac{1}{n} \#\left\{ j\leq n : \lambda_{j} > \frac{1}{2} \log\left( q/n\right) \right\} -   (n-q)/n \right|  . 
\end{align*}
For any positive integer $m+1\leq q\leq n$, define
\begin{equation}
\label{tria-lem-1}
p:=p(n,m,q,N)=\mathbb P \left( \left| \frac{1}{n} \#\left\{ j\leq n : \lambda_{j} > \frac{1}{2} \log\left(q/n\right) \right\} - (n-q)/n \right| \geq \frac{m}{ n} \right) .
\end{equation}
We have
\[\mathbb P\lr{S_{n,m}^c} \leq \sum_{q=m+1}^n p(n,m,q,N),\]
and the proof of Theorem \ref{T:global} therefore reduces to estimating the probabilities in this sum. To do this, we fix $n,m,q,N$ and observe that since $\lambda_j$ are decreasing, the event whose probability we've denoted by $p(n,m,q,N)$ is equal to 
$$ \left\{ \lambda_{n-q+m} > \frac{1}{2}  \log\left(q/n\right)  \right\}
\ \cup \   \left\{ \lambda_{n-q-m+1} \leq \frac{1}{2}  \log\left( q/n\right) \right\},$$
where we remind the reader that the second event is empty if $q\geq n-m+1.$ Again using the monotonicity of $\lambda_j$, this implies
$$  \left\{  \frac{1}{n}\sum_{j=n-q+1}^{n-q+m} \lambda_{j}> \frac{m}{2n}  \log\left( q/n\right)  \right\}  
\ \cup \   \left\{  \frac{1}{n} \sum_{j=n-q-m+1}^{n-q} \lambda_{j} \leq \frac{m}{2n}  \log\left(q/n\right) \right\}.$$
So, the probability $p(n,m,q,N)$ we seek to bound is itself bounded above by 
$$p_{1} + p_{2} := \mathbb P \left(    \frac{1}{n}\sum_{j=n-q+1}^{n-q+m} \lambda_{j}> \frac{m}{2n}  \log\left( q/n\right)    \right) + 
\mathbb P \left(   \frac{1}{n} \sum_{j=n-q-m+1}^{n-q} \lambda_{j} \leq \frac{m}{2n}  \log\left(q/n\right) \right) . $$
\noindent To estimate $p_{1}$ note that by Lemma \ref{L:digamma}, 
\[\frac{m}{2n}\log\lr{q/n} - \frac{1}{n}\sum_{j=n-q+1}^{n-q+m} \mu_{n,j}\geq \frac{(m-1)^2}{4nq}.\]
Hence, we obtain
\[p_1\leq \mathbb P\lr{\abs{\frac{1}{n}\sum_{j=n-q+1}^{n-q+m}\left\{\lambda_j- \mu_{n,j} \right\}}\geq \frac{(m-1)^2}{4nq}}.\]
We will bound the right hand side by using Theorem \ref{T:sums}. To do this, we must ensure that for $c_2$ sufficiently large, our assumption $N>c_2/\eps^2$ implies that for the constant $c_1$ in Theorem \ref{T:sums}, we have
\begin{equation}\label{E:apply-thm}
    \frac{(m-1)^2}{4nq}\geq c_1 \frac{n-q+m}{nN}\log\lr{\frac{en}{n-q+m}}.
\end{equation}
To check this, note that since $x\log(e/x)\leq 1$ for $x\in [0,1]$ this estimate holds as soon as 
\begin{equation}\label{E:inq1}
    N\geq c_1 \frac{4nq}{(m-1)^2}.
\end{equation}
Recall that by construction, we have 
\[q\leq n,\qquad (m-1)^2 \geq\frac{1}{2}m^2= \frac{1}{2}\eps^2 n^2.\]
Hence, \eqref{E:inq1} is satisfied once
\[N\geq 8c_1\eps^{-2},\]
as claimed. Thus, we may apply Theorem \ref{T:sums} to conclude that
\[p_1(n,m,q,N)\leq c_3\exp\lr{c_4\min\set{nNs, n^2Nsg_{n,k}(s)}},\qquad s=s(n,m,q)=\frac{(m-1)^2}{4nq}.\]
Since
\[\inf_{q=m+1,\ldots,n}s(n,m,q) \geq \frac{\eps^2}{8},\]
we find that
\[\sum_{q=m+1}^n p_1(m,n,q,N) \leq c_3\exp\lr{-c_4\min\set{nN\eps^2, n^2N\eps^2 g_{n,k}(\eps^2)} + \log(n)}\]
for some universal constants $c_3,c_4>0$. Further, note that by assumption, 
\[
nN\eps^2 > c_2 n > \log(n)
\]
as soon as $c_2>1.$ Hence, at the cost of replacing $c_4$ by a slightly larger constant $c_4'$, we find that 
\[\sum_{q=m+1}^n p_1(m,n,q,N) \leq c_3\exp\lr{-c_4'\min\set{nN\eps^2, n^2N\eps^2 g_{n,k}(\eps^2)}}.\]
Essentially the identical argument (but this time the upper bound from Lemma \ref{L:digamma}) implies that this same upper bound holds for $p_2$ as well, completing the proof of Theorem \ref{T:global}.\hfill $\square$

\section{Asymptotic Normality: Proof for Theorem \ref{T:normal}}\label{S:normal-outline}
Theorem \ref{T:normal} concerns
\[\Lambda_k = \lr{\lambda_1,\ldots, \lambda_k}=\lr{\lambda_1(X_{N,n}),\ldots, \lambda_k(X_{N,n})},\]
the vector of the first $k$ Lyapunov exponents of $X_{N,n}$. Our aim is to prove that there exist universal constants $C,C'>0$ so that once $N\geq Cn\log(n)$ we have
\begin{equation}\label{E:normal-goal}
    d\lr{\Lambda_k, \mathcal N\lr{\mu_{n,\leq k}, \Sigma_{n,k,N}}}\leq C'\lr{\frac{k^{7/2}n\log^2(n)\log(N/n)}{N}}^{1/2},\qquad \Sigma_{n,k,N}:=\frac{1}{N}\mathrm{Diag}(\sigma_{n,\leq k}^{2})
\end{equation}
where $\mu_{n,\leq k},\sigma_{n,\leq k}^{2}$ are the vectors of means and variances of $\lr{\frac{1}{2}\log\lr{\frac{1}{n}\chi_{n-m+1}^2},\, m=1,\ldots,k}$ (see \eqref{E:mu-sigma-def}) and $d$ is the distance function defined in \eqref{E:dist-def}. To prove \eqref{E:normal-goal}, we introduce
\[S_k = \lr{\lambda_1,\lambda_1+\lambda_2,\ldots, \lambda_1+\cdots+\lambda_k}^*\]
and note that
\begin{equation}\label{E:T-def}
S_k= T\Lambda_k,    
\end{equation}
where $T$ is a lower triangular matrix all of whose lower-triangular entries are equal to $1$. The explain the strategy for proving Theorem \ref{T:normal}, let us fix $\Theta=\theta_1\wedge\cdots \wedge \theta_k$, where $\set{\theta_j}$ form an orthonormal $k$-frame in $\R^n.$ For $1\leq m \leq k,$ we will abbreviate
\[\Theta_{\leq m}= \theta_1\wedge\cdots \wedge \theta_m.\]
The idea of the proof is to compare $S_k,\Lambda_k$ to their ``pointwise'' analogs
\[\widehat{S}_k: = \frac{1}{N}\lr{\log \norm{X_{N,n}(\Theta_{\leq 1})}, \ldots, \log \norm{X_{N,n}(\Theta_{\leq k})}}^*\]
and
\begin{equation}\label{E:T-def-0} \widehat{\Lambda}_k := T^{-1}\widehat{S}_k,
     \end{equation}
where $\Theta = \lr{\theta_1,\ldots,\theta_k}$ is any fixed collection of $k$ orthonormal vectors in $\R^n.$ Specifically, by Proposition \ref{Dist-lem} and the affine invariance \eqref{E:affine-inv} of $d$, we find that there exists $c_0>0$ so that for all $\delta>0$
\begin{align}
\notag    d\lr{\Lambda_k, \mathcal N\lr{\mu_{n,k}, \Sigma_{n,k}}} & = d\lr{S_k, \mathcal N\lr{T\mu_{n,k}, T\Sigma_{n,k,N} T^*}} \\
\notag    &\leq 3d\lr{\widehat{S}_k,\mathcal N\lr{T\mu_{n,k}, T\Sigma_{n,k,N} T^*}} +c_0\delta \norm{\lr{T\Sigma_{n,k} T^*}^{-1}}_{HS}^{1/2}\\
\notag &+ 2 \mathbb P\lr{\norm{S_k - \widehat{S}_k}>\delta}\\
\notag    &= 3d\lr{\widehat{\Lambda}_k,\mathcal N\lr{\mu_{n,k}, \Sigma_{n,k,N} }} +c_0\delta \norm{\lr{T\Sigma_{n,k,N} T^*}^{-1}}_{HS}^{1/2}\\
\label{E:KS-terms}&+ 2 \mathbb P\lr{\norm{S_k - \widehat{S}_k}>\delta}.
\end{align}
The remainder of the proof consists of bounding each of these three terms and then optimizing over $\delta$. To start, let us check that the first term in \eqref{E:KS-terms} is small:
\begin{lemma}\label{L:pointwise-guassian}
In distribution, 
\begin{equation}\label{E:pointwise-gaussian}
    \widehat{\Lambda}_k = \frac{1}{N}\sum_{i=1}^N \lr{Y_{i,1},\ldots, Y_{i,k}},
\end{equation}
where $\set{Y_{i,j},~1\leq i \leq N,\,\, 1\leq j\leq k}$ are independent with
\[Y_{i,j}\sim \frac{1}{2} \log \lr{\frac{1}{n}\chi_{n-j+1}^2}.\]
Consequently, by the multivariate central limit theorem, there exists $C>0$ so that
\begin{equation}\label{E:first-term}
    d\lr{\widehat{\Lambda}_k,\mathcal N\lr{\mu_{n,\leq k}, \Sigma_{n,k,N} }}\leq \frac{C k^{7/4}}{N^{1/2}}
\end{equation}
where $  \Sigma_{n,k,N} = \frac{1}{N} {\mathrm{ Diag}(\sigma_{n,\leq k }^{2} )}$, $ \sigma_{n,j}^{2} :=\Var\left[\frac{1}{2}\log\lr{\frac{1}{n}\chi_{n-j+1}^2}\right]$.
\end{lemma}
\begin{proof}
Fix integers $N,n\geq 1$ and $1\leq k \leq n$ and recall that $X_{N,n}=A_N\cdots A_1$ with $A_i$ iid $n\x n$ Gaussian matrices. Note that for each $1\leq m\leq k$, we have
\begin{equation}\label{E:vector-log}
    \log\norm{X_{N,n}(\Theta_{\leq m})} = \sum_{i=1}^N \log \norm{A_i(\Theta_{\leq m}^{(i)})},
\end{equation}
where
\[
\Theta_{\leq m}^{(1)}= \Theta_{\leq m},\qquad \Theta_{\leq m}^{(i+1)} = \frac{A_i\lr{\Theta_{\leq m}^{(i)}}}{\norm{A_i\lr{\Theta_{\leq m}^{(i)}}}}.
\]
Repeatedly applying Lemma \ref{L:haar-flags}, we therefore conclude that in distribution
\[\widehat{S}_k=\frac{1}{N}\sum_{i=1}^N \lr{\log\norm{A_i(\Theta_{\leq 1})},\ldots, \log \norm{A_i(\Theta_{\leq k})}}^*\]
is equal to a sum of iid random vectors. Thus, using the definition \eqref{E:T-def-0} of $\widehat{\Lambda}_k$, we find that in distribution
\[\widehat{\Lambda}_k=\frac{1}{N}\sum_{i=1}^N\widehat{\Lambda}_{k,i},\qquad \widehat{\Lambda}_{k,i}:= T^{-1}\lr{\log\norm{A_i(\Theta_{\leq 1})},\ldots, \log \norm{A_i(\Theta_{\leq k})}}^*,\]
where we recall that $T$ is a lower triangular matrix with all lower triangular entries equal to $1.$ Namely,
\[T = \lr{\begin{array}{ccccc}
    1 & 0 & 0&\cdots & 0  \\
    1 &  1 & 0& \cdots & 0  \\
    \vdots & \cdots & \ddots &\ddots &\vdots\\
    1 & \cdots & 1 & 1& 0  \\
    1 & \cdots & 1 & 1& 1   \\\end{array}},\qquad T^{-1}= \lr{\begin{array}{ccccc}
    1 & 0 & 0&\cdots & 0  \\
    -1 &  1 & 0& \cdots & 0  \\
    \vdots & \ddots & \ddots &\ddots &\vdots\\
    0 & \cdots & -1 & 1& 0   \\
    0 & \cdots & 0 & -1& 1   \\
\end{array}}.\]
Note that $\set{\widehat{\Lambda}_{k,i},\, m=1,\ldots, k}$ are independent collection for different $i$. Next, the $m^{th}$ component of $\widehat{\Lambda}_{k,i}$ is
\begin{align}
\label{E:components}
\lr{\widehat{\Lambda}_{k,i}}_m=\log\norm{A_i(\Theta_{\leq m})} -\log\norm{A_i(\Theta_{\leq m-1})} = \log \norm{\frac{A_i(\Theta_{\leq m-1})}{\norm{A_i(\Theta_{\leq m-1})}}\wedge A\theta_m}
\end{align}
Since $\set{\theta_i}$ are orthonormal, the collection $\set{A\theta_i}$ are iid Gaussians. In particular, we see that $A\theta_m$ is independent of $\set{A(\Theta_{\leq j}), 1\leq j\leq m-1}.$ Also, by Lemma \ref{L:polar}, the following collections of random variables are independent:
\[\set{\norm{A(\Theta_{\leq 1})},\ldots, \norm{A(\Theta_{\leq m-1})}},\qquad \set{\frac{A(\Theta_{\leq 1})}{\norm{A(\Theta_{\leq 1})}}, \ldots,\frac{A(\Theta_{\leq  m-1})}{\norm{A(\Theta_{\leq m-1})}}}.\]
The left hand side of relation \eqref{E:components} shows that the $1,\ldots,m-1^{st}$ components of $\Lambda_{k,i}$ depend only $\set{\norm{A(\Theta_{\leq j})},\, j=1,\ldots, m-1}$, whereas the right hand side of \eqref{E:components} shows that the $m^{th}$ component of $\Lambda_{k,i}$ depends only on $A(\Theta_{\leq m-1})/\norm{A(\Theta_{\leq m-1})}$ and on $A\theta_m$. Therefore, the $m^{th}$ component of $\widehat{\Lambda}_{k,i}$ is independent of all the previous components. Proceeding in this way for $m=k,k-1,\ldots,1$, we find that the components of $\widehat{\Lambda}_{k,i}$ are independent. Finally, let us denote by $\Pi_{\leq m-1}$ the orthogonal projection onto the orthogonal complement of the span of $\set{\theta_1,\ldots,\theta_{m-1}}.$ We have by Lemma \ref{lem-Grass-2} that in distribution
\[\lr{\widehat{\Lambda}_{k,i}}_m = \log \norm{ \Pi_{\leq m-1}(A\theta_m)}.\]
Note that $A\theta_m$ is independent of $\Pi_{\leq m-1}$ since the latter depends only on $A\theta_1,\ldots,A\theta_{m-1}$. Hence, we have the following equality in distribution:
\[\lr{\widehat{\Lambda}_{k,i}}_m = \frac{1}{2}\log\lr{ \frac{1}{n}\chi_{n-m+1}^2}.\]
This completes the proof of \eqref{E:pointwise-gaussian}. To conclude \eqref{E:first-term}, we apply the multivariate CLT (Theorem \ref{MultiCLT-theo}) to 
\[\widehat{\Lambda}_k-\E{\widehat{\Lambda}_k}=\sum_{i=1}^N \frac{1}{N}\lr{\widehat{\Lambda}_{k,i} -\mu_{n,\leq k}}.\]
Since  the covariance matrix of $ (Y_{i,1}, \cdots , Y_{i,k} ) $ is ${\mathrm{ Diag}(\sigma_{n,\leq k }^{2} )}$ by independence we have that $C:= {\rm Cov} ( \widehat{\Lambda}_{K}) := \frac{1}{N} {\rm Diag} ( \sigma_{n,1}^{2}, \cdots , \sigma_{n,k}^{2})$.  Recall that $ \beta_{i} := \mathbb E \| C^{-\frac{1}{2}} (\overline{Y}_{i,1}, \cdots , \overline{Y}_{i,k})\|_{2}^{3} $. It is not difficult to check that $ \log\chi_{m}^{2} $ is a $\log$-concave random variable (i.e. its density is a log-concave function). Moreover, since $ \sigma_{n,j}^{-1} \overline{Y}_{i,j}$ have mean zero and variance $1$, $ D:=(  \sigma_{i,1}^{-1} \overline{Y}_{i,1}, \cdots, \sigma_{i,k}^{-1} \overline{Y}_{i,k})$ is a log-concave random vector in $\mathbb R^{k}$ with covariance matrix equals to the identity. Therefore $\mathbb E \|D\|_{2}^{2} = k $. It is known that the Euclidean norm of such vectors satisfies a reverse H\"older inequality with a universal constant,  and in particular (see e.g. \cite{paouris2006concentration} or [\cite{artstein2015asymptotic} Theorem 10.4.6] for a stronger result) that 
$$ \left( \mathbb E \| D\|_{2}^{3} \right)^{\frac{1}{3}} \leq  C \left( \mathbb E \| D\|_{2}^{2} \right)^{\frac{1}{2}}= C\sqrt{k} , $$
where $C>0$ is an absolute constant. So,
$$ \beta_{i} =\frac{1}{N^{\frac{3}{2}}}  \mathbb E \| D \|_{2}^{3} \leq \frac{C^{3} k^{\frac{3}{2}}  }{N^{\frac{3}{2}}} \ 1\leq i \leq N.$$
Therefore, 
$$ \beta:= \sum_{j=1}^{N} \beta_{j} \leq \frac{C^{3} k^{\frac{3}{2}}  }{N^{\frac{1}{2}}}  $$
and we conclude that there exists an absolute constant $c>0$ so that
\[d(\widehat{\Lambda}_k, \mathcal N(\mu_{n,\leq k}, 
\Sigma_{n,k}))\leq c k^{7/4}N^{-1/2}. \]
\end{proof}
\noindent Having bounded the first term in \eqref{E:KS-terms}, we write 
\[\lr{T\Sigma_{n,k,N} T^*}^{-1} = \lr{T^*}^{-1}\Sigma_{n,k,N}^{-1}T^{-1}\]
and bound the second term using that the matrix $\Sigma$ is diagonal and that $T^{-1}$ a bi-diagonal:
\begin{lemma}\label{L:HS-est}
There exists $C>0$ so that
\begin{equation}\label{E:second-term}
    \norm{\lr{T\Sigma_{n,k,N} T^*}^{-1}}_{HS}^{1/2}\leq Ck^{1/4}(nN)^{1/2}.
\end{equation}
\end{lemma}
\begin{proof}
We have
\[T = \lr{\begin{array}{ccccc}
    1 & 0 & 0&\cdots & 0  \\
    1 &  1 & 0& \cdots & 0  \\
    \vdots & \cdots & \ddots &\ddots &\vdots\\
    1 & \cdots & 1 & 1& 0  \\
    1 & \cdots & 1 & 1& 1   \\\end{array}},\qquad T^{-1}= \lr{\begin{array}{ccccc}
    1 & 0 & 0&\cdots & 0  \\
    -1 &  1 & 0& \cdots & 0  \\
    \vdots & \ddots & \ddots &\ddots &\vdots\\
    0 & \cdots & -1 & 1& 0   \\
    0 & \cdots & 0 & -1& 1   \\
\end{array}}.\]
Thus, recalling that
\[\Sigma_{n,k,N}= \frac{1}{N}\mathrm{Diag}\lr{\sigma_{n,\leq k}}=\frac{1}{N}\lr{\sigma_{n,1}^{2},\ldots, \sigma_{n,k}^{2}},\]
we find
\[(T^*)^{-1}\Sigma_{n,k,N}^{-1}T^{-1} =  N\lr{\begin{array}{ccccc}
    \sigma_{n,1}^{-2}+\sigma_{n,2}^{-2} & - \sigma_{n,2}^{-2}& 0&\cdots & 0  \\
    -\sigma_{n,2}^{-2} &   \sigma_{n,2}^{-2}+\sigma_{n,3}^{-2} & -\sigma_{n,3}^{-2}& \cdots & 0  \\
    \vdots & \cdots & \ddots &\ddots &\vdots\\
    0 & \cdots & -\sigma_{n,k-1}^{-2} &\sigma_{n,k-1}^{-2}+\sigma_{n,k}^{-2} & -\sigma_{n,k}^{-2}  \\
    0 & \cdots & 0 &-\sigma_{n,k}^{-2}  & \sigma_{n,k}^{-2}   \\\end{array}}\]
Hence, using \eqref{E:sigma-est}, we find that for some $C>0$
\[\norm{(T^*)^{-1}\Sigma_{n,k,N}^{-1}T^{-1}}_{HS} \leq  2N \lr{\sum_{j=1}^k \sigma_{n,k,j}^{-4}}^{1/2}\leq CN\lr{\sum_{j=1}^k (n-k+1)^2}^{1/2}\leq CNnk^{1/2},\]
and Lemma \ref{L:HS-est} follows.
\end{proof}
\noindent Thus far, combining the previous two Lemma with \eqref{E:KS-terms}, we've shown that 
\begin{align}
\label{E:first-two}d\lr{\Lambda_k, \mathcal N\lr{\mu_{n,k}, \Sigma_{n,k}}} &\leq \frac{Ck^{7/4}}{N^{1/2}} +c_0\delta k^{1/4}(nN)^{1/2}+ 2 \mathbb P\lr{\norm{S_k - \widehat{S}_k}>\delta}.    
\end{align}
So it remains to estimate
\[
  \mathbb P\lr{\norm{S_k - \widehat{S}_k}_2\geq\delta} 
\]
and optimize over $\delta$. To do this, write $S_{k,j}, \widehat{S}_{k,j}$ for the $j^{th}$ components of $S_k,\widehat{S}_k$. By   \eqref{conc-by-sb}, there exists $C>0$ so that for $ 1\leq j \leq k \leq n$, 
\begin{equation*}
\mathbb P \left( | S_{k,j}  -\widehat{S}_{k,j}|\geq  s   \right) \leq 2e^{ - sN/4 } ,\qquad s\geq C\frac{ j}{N} \log\lr{ \frac{en}{j}}.
\end{equation*}
For any collection positive real numbers $\delta_j>C\frac{j}{N}\log\lr{\frac{en}{j}}$ we therefore have,
\[
  \mathbb P\lr{\norm{S_k - \widehat{S}_k}_2\geq\lr{\sum_{j=1}^k \delta_j^2}^{1/2}}\leq \sum_{j=1}^k \mathbb P \left( | S_{k,j}  -\widehat{S}_{k,j}|\geq  \delta_j \right) \leq 2\sum_{j=1}^ke^{-\delta_j N/4 }.
\]
Setting
\[
\delta_j:=\frac{Cj}{N}\log\lr{\frac{en}{j}} \log\lr{\frac{N}{n}},
\]
for a sufficiently large constant $C$ we find
\[
  \mathbb P\lr{\abs{S_{k,j} - \widehat{S}_{k,j}}\geq\delta_j}\leq 2e^{-Cj\log(en /j )\log(N/n)}\leq 2(n/N)^{j/2}.
\]
Hence, as soon as $N> n$, we have
\[
\mathbb P\lr{\norm{S_k - \widehat{S}_k}_2\geq\delta}\leq C\lr{\frac{n}{N}}^{1/2}
\]
where
\[
\delta :=\lr{\sum_{j=1}^{k}\delta_j^2}^{1/2} \leq \frac{C k^{3/2}\log(n)\log(N/n)}{N}.
\]
In conjunction with \eqref{E:first-two} yields
\begin{align*}
d\lr{\Lambda_k, \mathcal N\lr{\mu_{n,k}, \Sigma_{n,k}}} &\leq \frac{Ck^{7/4}}{N^{1/2}} +\lr{\frac{Ck^{7/2}n \log^2(n)\log^2(N/n)}{N}}^{1/2} + C\lr{\frac{n}{N}}^{1/2}\\  
&\leq \lr{\frac{4Ck^{7/2}n \log^2(n)\log^2(N/n)}{N}}^{1/2},
\end{align*}
as claimed.

\bibliographystyle{alpha}
  \bibliography{bibliography}

\end{document}